\pdfoutput=1
\documentclass[10pt]{article}

\usepackage{fullpage}

\usepackage{authblk}

\makeatletter
\def\thanks#1{\protected@xdef\@thanks{\@thanks\protect\footnotetext{#1}}}
\makeatother

\usepackage{microtype}

\usepackage{amssymb,amsthm}
\usepackage{amsmath}
\usepackage{stmaryrd}
\usepackage{wasysym}

\usepackage{mathpartir}
\usepackage{tikz-cd}

\usepackage{enumitem} 
\newlist{casesp}{enumerate}{1}
\setlist[casesp]{align=left, 
                 listparindent=\parindent, 
                 parsep=\parskip, 
                 font=\normalfont\it, 
                 leftmargin=10pt, 
                 labelwidth=10pt, 
                 itemindent=.4em,
                 itemsep = 3\itemsep,
                 labelsep=.4em,
                 partopsep=0pt, 
                 }
\setlist[casesp,1]{label=Case:,ref=\arabic*}

\usepackage{graphicx}
\usepackage{scalerel}
\usepackage{fontawesome}

\usepackage{xcolor}
\definecolor{darkgreen}{rgb}{0,0.45,0}

\usepackage[status=draft,inline,nomargin]{fixme}
\FXRegisterAuthor{mvr}{anmvr}{\color{blue}MVR}

\usepackage[style=alphabetic, maxbibnames=99, maxalphanames=99, mincrossrefs=99]{biblatex}
\usepackage{hyperref}
\usepackage[capitalise,noabbrev]{cleveref}
\hypersetup{colorlinks,allcolors=[rgb]{0.1,0.1,0.4}}

\newtheorem{theorem}{Theorem}[section]
\newtheorem{proposition}[theorem]{Proposition}
\newtheorem{lemma}[theorem]{Lemma}
\newtheorem{corollary}[theorem]{Corollary}

\theoremstyle{definition}
\newtheorem{definition}[theorem]{Definition}
\newtheorem{remark}[theorem]{Remark}

\let\oldequiv\equiv%
\renewcommand{\equiv}{\simeq}
\newcommand{\defeq}{\oldequiv}
\newcommand{\ndefeq}{\not\defeq}

\newcommand{\rulen}[1]{\textsc{#1}}
\newcommand{\yields}{\vdash}

\newcommand{\judge}{\mathcal{J}}

\newcommand{\ctx}{\,\,\mathsf{ctx}}
\newcommand{\type}{\,\,\mathsf{type}}
\newcommand{\tele}{\,\,\mathsf{tele}}

\newcommand{\rawterm}{\,\,\mathsf{rawterm}}

\newcommand{\Idsym}{\mathsf{Id}}

\newcommand*{\univ}{\mathcal{U}}
\newcommand*{\NN}{\mathbb{N}}

\newcommand*{\RR}{\mathbb{R}}
\newcommand*{\id}{\mathsf{id}}
\newcommand*{\proj}{\mathsf{pr}}

\newcommand*{\inl}{\mathsf{inl}}
\newcommand*{\inr}{\mathsf{inr}}

\newcommand*{\refl}[1]{\mathsf{refl}_{#1}}

\makeatletter \def\smsym{\sum}
\newcommand{\@thesum}[1]{\smsym_{(#1)}}
\newcommand{\sm}[1]{\@ifnextchar\bgroup{\@sm{#1}\sm}{\@sm{#1}}}
\newcommand{\@sm}[1]{\mathchoice{{\textstyle\@thesum{#1}}}{\@thesum{#1}}{\@thesum{#1}}{\@thesum{#1}}}
\def\prdsym{\prod}
\newcommand{\@theprd}[1]{\prdsym_{(#1)}}
\newcommand{\prd}[1]{\@ifnextchar\bgroup{\@prd{#1}\prd}{\@prd{#1}}}
\newcommand{\@prd}[1]{\mathchoice{{\textstyle\@theprd{#1}}}{\@theprd{#1}}{\@theprd{#1}}{\@theprd{#1}}}
\makeatother

\newcommand*{\lolli}{\multimap}

\newcommand{\lock}{\mathchoice {\scalebox{0.8}{\text{\faLock}}}
  {\scalebox{0.8}{\text{\faLock}}} {\scalebox{0.5}{\text{\faLock}}}
  {\scalebox{0.4}{\text{\faLock}}} }

\newcommand{\key}{\mathchoice
  {\scalebox{0.8}{\text{\faKey}}} {\scalebox{0.8}{\text{\faKey}}}
  {\scalebox{0.5}{\text{\faKey}}} {\scalebox{0.4}{\text{\faKey}}} }

\newcommand{\rbindsym}{\raisebox{-0.5pt}{\scalerel*{\reflectbox{\rotatebox[origin=c]{185}{$\lambda$}}}{\lambda}}}
\newcommand{\rbind}[1]{\rbindsym{} #1.}

\newcommand{\Tiny}{\mathbb{T}}
\newcommand{\lockn}[1]{\mathcal{#1}}
\newcommand{\locksin}[1]{\mathsf{locks}(#1)}

\newcommand{\ctxlocke}[1]{\lock_{#1}}
\newcommand{\ctxlock}[1]{\ctxlocke{\lockn{#1}}}
\newcommand{\stubra}[1]{\llbracket #1 \rrbracket}
\newcommand{\admbra}[1]{[ #1 ]}
\newcommand{\subkeyoe}[2]{#1/\key_{#2}}
\newcommand{\subkeyo}[2]{\subkeyoe{#1}{\lockn{#2}}}

\newcommand{\subkey}[2]{\admbra{\subkeyo{#1}{#2}}}
\newcommand{\subkeyeann}[5]{\admbra{\key(#3; #1/#2; #4; #5)}}
\newcommand{\subkeyann}[5]{\subkeyeann{#1}{\lockn{#2}}{#3}{#4}{#5}}

\newcommand{\substucke}[2]{\stubra{#1/\key_{#2}}}
\newcommand{\substuck}[2]{\substucke{#1}{\lockn{#2}}}
\newcommand{\sublock}[2]{\admbra{\rbindsym{} #1. /\lock_{\lockn{#2}} }}

\newcommand{\unit}[3]{\mathsf{unit}_{#1, #2, \lockn{#3}}}
\newcommand{\counit}[3]{\mathsf{counit}_{#1, \lockn{#2}, #3}}

\newcommand{\extract}{\mathsf{e}}

\newcommand{\rformsym}{\surd}
\newcommand{\rformslicesym}[1]{\surd_{\mkern-6mu/#1}}
\newcommand{\rforme}[2]{\ThisStyle{\raisebox{0.07em}{$\SavedStyle \rformsym_{\hspace{-0.25em}#1}$}} #2}
\newcommand{\rform}[2]{\rforme{\lockn{#1}}{#2}}
\newcommand{\rformu}[1]{\ThisStyle{\raisebox{0.07em}{$\SavedStyle \rformsym$}} #1}
\newcommand{\rintroe}[2]{\lock_{#1}. #2}
\newcommand{\rintro}[2]{\rintroe{\lockn{#1}}{#2}}

\newcommand{\relim}[1]{\rbindsym #1}

\newcommand{\rdepform}[2]{{#2}^{1/#1}}

\newcommand{\rget}[1]{{#1}_{\downharpoonleft}}

\newcommand{\PSh}{\mathsf{PSh}}
\newcommand{\op}{\mathsf{op}}
\newcommand{\yo}{\mathsf{y}}
\newcommand{\Set}{\mathsf{Set}}

\newcommand{\C}{\mathcal{C}}

\newcommand{\univfib}{\mathsf{U}}
\newcommand{\Int}{\mathbb{I}}
\DeclareMathOperator{\isFib}{\mathsf{isFib}}
\DeclareMathOperator{\Fib}{\mathsf{Fib}}
\DeclareMathOperator{\isNewFib}{\mathsf{isNewFib}}
\DeclareMathOperator{\NewFib}{\mathsf{NewFib}}
\newcommand{\Compstr}{\mathsf{C}}

\title{A Type Theory with a Tiny Object}
\author{Mitchell Riley\thanks{The author is grateful for the support of Tamkeen under the
    NYU Abu Dhabi Research Institute grant CG008.}}
\affil{\small New York University Abu Dhabi \\ \href{mailto:mitchell.v.riley@nyu.edu}{mitchell.v.riley@nyu.edu}}
\date{\today}

\begin{document}
\maketitle

\begin{abstract}
  We present an extension of Martin-Löf Type Theory that contains a
  tiny object; a type for which there is a \emph{right} adjoint to the
  formation of function types as well as the expected left adjoint.
  We demonstrate the practicality of this type theory by proving
  various properties related to tininess internally and suggest a few
  potential applications.
\end{abstract}

\tableofcontents

\section{Introduction}

Tiny objects are central to Lawvere's account of differential forms in
synthetic differential geometry~\cite{lawvere:towards}; a tiny object $\Tiny$ in a
category $\C$ is one with the amazing property that the internal hom
functor ${(\Tiny \to -) : \C \to \C}$ has a right adjoint
${\rformu - : \C \to \C}$. In synthetic differential geometry, the
\emph{infinitesimal interval} defined by
${D :\defeq \{ x : \mathbb{R} \mid x^2 = 0 \}}$ is one such object. This
object corepresents tangent spaces so that $D \to X$ is the tangent
space of $X$, and elements of $(D \to X) \to \RR$ are therefore
(not-necessarily linear) 1-forms on $X$. By tininess, such a form
corresponds to an element of $X \to \rformu \RR$, that is, 1-forms are
simply functions on $X$ valued in a highly non-standard space.
Additional properties like linearity, closedness, etc.\ may be imposed
by carving out an appropriate subobject of $\rformu \RR$. The notion
of tininess was anticipated by~\cite{kock-reyes:manifolds,
  lawvere:categorical-dynamics} but its utility was first made
explicit in~\cite[\S 3]{lawvere:towards}, and it appears in much of
the ensuing work on synthetic differential
geometry~\cite{lawvere:laws-of-motion, lawvere:outline-sdg,
  lawvere:microphysics, lawvere:adjoints, lawvere:euler-vindicated,
  kock-reyes:differential-equations, kock-reyes:fractional, kock:sdg}.

Tininess is simultaneously unusual and abundant. Unusual because, in
the category of sets, only the singletons are tiny. This is easy to
check: $(\Tiny \to -)$ can only preserve coproducts when $\Tiny \cong 1$.
And abundant because, for any category $\C$ with finite products, the
representable presheaves on $\C$ are all tiny. An explicit formula for
the right adjoint to $(\Tiny \to -)$ with $\Tiny$ representable is
$\rformu{Y}(c) :\defeq \PSh(\C)(\Tiny \to \yo c, Y)$, as may be
verified by a little (co)end calculus. In favourable cases
representable sheaves are also tiny~\cite[Appendix
4]{moerdijk-reyes:book}, but this is less common.


In this paper we tackle a challenge set by Lawvere, to produce a
formal system for working with tiny objects: ``This possibility does
not seem to have been contemplated by combinatory logic; the formalism
should be extended to enable treatment of so basic a
situation''~\cite[Section 3]{lawvere:adjoints}. We describe an
extension of Martin-Löf Type Theory that makes a fixed type $\Tiny$
tiny by introducing a type former $\rformu$ for the amazing right
adjoint to $(\Tiny \to -)$. The situation is a little subtle. Freyd
proves that for any tiny object $\Tiny$ in a topos, the pullback of
$\Tiny$ to any slice is also tiny~\cite[Theorem 1.4]{yetter:tiny}. But
the right adjoint witnessing this tininess is \emph{not} stable under
base-change. We therefore have to add to the context structure of type
theory in order to support this operation.

We add $\rformu$ as a `Fitch style'
modality~\cite{clouston:fitch-style, drats}, where the type former
$\rformu$ is made right adjoint to an operation on contexts. Such
modalities are particularly nice when they are a \emph{double}
right adjoint, and \textsf{FitchTT}~\cite{fitchtt} is designed for adding
modalities of this kind to MLTT\@. We are in such a situation, of
course, because
\[{(- \times \Tiny) \dashv (\Tiny \to -) \dashv \rformu -},\] and so
we could use~\cite{fitchtt} directly to produce a type theory.

But there is a special feature of $\rformu$ which impels us to create
a specialised theory for it. Specifically, the leftmost adjoint
$(- \times \Tiny)$ already exists as an operation on contexts: it is
simply context extension with $i : \Tiny$. We allow $\Tiny$ to be an
\emph{ordinary type}, rather than a pre-type or special piece of
syntax. This flexibility is essential for applications in synthetic
differential geometry where many important tiny objects may be
constructed internally as ordinary types. The only new context former
needed is something corresponding to $(\Tiny \to -)$, which we write
as context extension with a $\ctxlock{}$.

To make $(-,i:\Tiny) \dashv (-,\lock)$ on contexts, we need unit
$\Gamma \to (\Gamma,i : \Tiny, \lock)$ and counit
$(\Gamma,\lock,i : \Tiny) \to \Gamma$ substitutions. In~\cite{fitchtt}
these are added axiomatically, and the type theory is presented in a
variable-free CwF style where these explicit substitutions are pushed
around manually. The downside of the approach is that figuring out how
to use a variable could be potentially challenging: one may have to
devise by hand an explicit substitution that extracts the variable
from the context. One of our aims is to give a fully explicit variable
rule which builds in the normal forms of these `stuck substitutions'.

The reason this gets interesting is that, because $\Tiny$ is an
ordinary type, we can substitute \emph{any} term $t : \Tiny$ for $i$
in the counit substitution. And so, the admissible counit rule is
parameterised by a genuine term $t : \Tiny$. The counit rule commutes
with all type and term formers and becomes stuck on the ordinary
variable rule, so every use of a variable in this theory will have
(possibly many) attached terms of $\Tiny$ corresponding to these stuck
counits.

\paragraph{Related Work.} 
In~\cite{lops}, a tiny interval $\mathbb{I}$ is used to construct,
internal to a 1-topos, a universe that classifies fibrations. This is
performed in `crisp type theory', a fragment of Shulman's spatial type
theory~\cite{mike:real-cohesive-hott}. The $\rformu$ type former is
described by a collection of axioms, and the fact that the adjunction
is external is enforced by requiring the inputs to these axioms to be
`crisp', roughly, protected by a use of the global sections/discrete
inclusion modality $\flat$. Internally, this manifests as an
equivalence
${\flat(A \to \rformu B) \equiv \flat((\Tiny \to A) \to B)}$. An
alternative axiomatisation that is coherent for higher types is given
in~\cite[Appendix A]{david:microlinear}, by instead asserting for each
crisp type $B$ a counit map $(\Tiny \to \rformu B) \to B$ such that
the induced map
${\flat(A \to \rformu B) \to \flat((\Tiny \to A) \to B)}$ is an
equivalence.

These equivalences guarded by $\flat$ are more restrictive than
necessary: we will prove an equivalence
${(A \to \rformu B) \equiv \rformu ((\Tiny \to A) \to B)}$, where $A$
and $B$ do not have to be `global types' (but the dependency of $B$ is
still somewhat restricted).

The most comparable type theory to ours is presented
in~\cite{transpension}. The authors give rules for a right adjoint to
the \emph{dependent product}
$\Pi_\Tiny : \mathcal{E}/\Tiny \to \mathcal{E}$ rather than the
non-dependent function type, an operation which they call
\emph{transpension}. It is a theorem of Freyd~\cite[Proposition
1.2]{yetter:tiny} that this is equivalent to the apparently weaker
notion of tininess. Because we target the non-dependent function type,
our new judgemental structures and rules are much simpler than those
for transpension, and let us maintain admissibility of substitution
and (conjecturally) normalisation; for transpension it is unclear
whether these are achievable.

The context extension with the tiny type is also treated specially,
similar to accounts of internal
parametricity~\cite{cavallo-harper:parametricity-for-ctt,
  cavallo:thesis}~\cite[Section 5]{fitchtt} where the
$\Gamma, i : \Tiny$ context extension is a special piece of syntax,
and so only special `terms' may be substituted for it.

A benefit of these previous systems is that the precise notion of
tininess is easier to tweak. For example, whenever $\C$ is monoidal,
representable presheaves on $\C$ are \emph{monoidally tiny} in that
$(\Tiny \lolli -)$ has a right adjoint, where $\lolli$ denotes the Day
internal hom. This non-cartesian tininess is central to forthcoming
work on Higher Observational Type
Theory~\cite{mike:observational-day3}, whose intended semantics are in
a non-cartesian cube category.

\paragraph{Contributions.}

\begin{itemize}
\item In~\cref{sec:contexts}, we introduce the new context structure
  required to support the type former and in~\cref{sec:type} the
  type-former itself.
\item In~\cref{sec:constructions}, we give a couple of useful
  constructions that are definable internally. \cref{sec:induction}
  shows that we can derive induction principles for functions out of
  $\Tiny$, and \cref{sec:transpension} recovers the transpension
  operation from our non-dependent right adjoint, providing a new
  construction that applies in the univalent setting.
\item In~\cref{sec:applications}, we discuss some potential
  applications of the type theory and a prototype type-checker
  utilising an extension of the usual normalisation by evaluation
  algorithm.
\end{itemize}

\paragraph{Terminology.} There is a constellation of terminology
around tininess and similar notions, with ``amazingly tiny'',
``atomic'', ``infinitesimal'', ``satisfying the ATOM property'',
``internally/externally projective'' and ``small-projective'' all used
to refer to one of $\C(\Tiny, -) : \C \to \Set$ or
$(\Tiny \to -) : \C \to \C$ preserving epimorphisms, finite colimits,
colimits or having a right adjoint. With additional assumptions many
of these notions coincide, but we follow~\cite{yetter:tiny} in
settling on simply ``tiny'' as the name for objects whose internal hom
has a right adjoint. The other properties play no role in this work.
There are also various notations for the right adjoint, including
$(-)_\Tiny$~\cite{yetter:tiny} and
$(-)^{1/\Tiny}$~\cite{lawvere:laws-of-motion}, and even
$\nabla_\Tiny$~\cite{yetter:tiny} and ${\between}-$ for
transpension~\cite{transpension}. We follow~\cite{lops} in using
$\rformu -$.

\paragraph{Acknowledgements.}

Thank you to David Jaz Myers, Andreas Nuyts and Jon Sterling for
helpful comments and suggestions. The author is grateful for the
support of Tamkeen under the NYU Abu Dhabi Research Institute grant
CG008.


\section{Contexts and Variables}\label{sec:contexts}

We take as our starting point the bare judgements of Martin-Löf Type
Theory. When using our theory we will also assume dependent sums,
products, intensional identity types and an infinite hierarchy of
universes, but the extension of type theory with a tiny object does
not require the presence of any other type-formers.

The $\rformu$ type-former will be added by making it right adjoint to
a judgemental version of ``$\Tiny \to \Gamma$'' which we write as
$\Gamma, \ctxlock{L}$. Here, to refer to it later, the lock is
annotated with a `lock name' $\lockn{L}$. This first section will be
spent adding the necessary rules for the context $\Gamma, \ctxlock{L}$
to behave like functions into $\Gamma$.

A simple way to achieve this would be to assert axiomatic unit and
counit substitutions for the $(- \times \Tiny) \dashv (\Tiny \to -)$
adjunction as in the following:
\begin{mathpar}
  \inferrule*[left=unit-sub]
  {\Gamma \ctx}
  {\Gamma \yields \unit{\Gamma}{i}{L} : \Gamma, i : \Tiny, \ctxlock{L}}
  \and
  \inferrule*[left=counit-sub]
  {\Gamma \ctx}
  {\Gamma, \ctxlock{L}, i : \Tiny \yields \counit{\Gamma}{L}{i} : \Gamma}
\end{mathpar}
together with equations that explain how these explicit substitutions
are pushed around and annihilated.

We can do better, however, and give normal forms for the placement of
these explicit substitutions. The examples in later sections show that
it is quite feasible to work in the resulting type theory by hand.

Besides the context lock, we only need one additional base rule: a
modified version of the variable rule that builds in stuck instances
of the counit substitution. The rule is completely structural and
independent of the rules for types (besides the existence of $\Tiny$).
There are two new admissible rules, corresponding to precomposition
with the counit or unit substitutions.

We now describe the additions to MLTT one rule at a time.

\paragraph{The Tiny Type.} There is a closed type $\Tiny$.
\begin{mathpar}
  \inferrule*[left=tiny-form]{~}{\Gamma \yields \Tiny \type}
\end{mathpar}
In applications, this will typically be an already-existing closed
type rather than a new one asserted axiomatically.

\paragraph{Context Locks.} There is a special context extension,
\begin{mathpar}
  \inferrule*[left=ctx-lock]{\Gamma \ctx}{\Gamma, \ctxlock{L} \ctx}
  \and
\end{mathpar}
to be thought of as $\Tiny \to \Gamma$. We call $\lockn{L}$ a `lock
name', all lock names in a context are unique.

\paragraph{The Counit.} Because a locked context represents functions
into what comes before the lock, we can use variables to the left of
a lock if we can provide an argument to the function: this
corresponds to precomposition with the counit substitution, with a
substitution for $\Tiny$ and some contractions built-in.

The simplest situation we can encounter is
\begin{mathpar}
  \inferrule*[left=counit?,fraction={-{\,-\,}-}]
  {\Gamma \yields a : A \and \Gamma, \ctxlock{L}, \Gamma' \yields t : \Tiny \and \lock \notin \Gamma'}
  {\Gamma, \ctxlock{L}, \Gamma' \yields a\subkey{t}{L} : A\subkey{t}{L}}
\end{mathpar}
corresponding (non-dependently) to the composite
\begin{align*}
  (\Tiny \to \Gamma) \times \Gamma'
  \xrightarrow{[\id, t]} (\Tiny \to \Gamma) \times \Gamma' \times
  \Tiny
  \xrightarrow{\proj} (\Tiny \to \Gamma) \times \Tiny
  \xrightarrow{\varepsilon} \Gamma
  \xrightarrow{a} A
\end{align*}

The new piece of term syntax is a stuck instance of this
$\subkey{t}{L}$ rule which we build into the variable rule below.
To distinguish the admissible rule from the stuck rule, we will
write the admissible rule as $\subkey{t}{L}$ and the actual stuck
syntax as $\substuck{t}{L}$. Roughly, the admissible $\subkey{t}{L}$
will add a $\substuck{t}{L}$ to every \emph{free} variable usage in
$a$, much like the underlining operation in the type theory for
$\natural$~\cite{rfl:spec-note}.

We need to generalise this rule in a couple of ways. First, we need
to allow an additional telescope to be carried along, so that we may
go under binders in the term $a : A$.

Additionally, we may need to apply the counit to multiple locks
simultaneously. To facilitate this, for a context
$\Gamma, \Gamma' \ctx$ let $\locksin{\Gamma'}$ denote the (ordered)
list of context locks
$\ctxlocke{\lockn{L}_1}, \dots, \ctxlocke{\lockn{L}_n}$ that appear
in the telescope $\Gamma'$.

The generalised rule is then:
\begin{mathpar}
  \inferrule*[left=counit-tele,fraction={-{\,-\,}-}]
  {\Gamma \yields \Gamma'' \tele \and \Gamma, \Gamma' \yields t_i : \Tiny
    \text{ for } \lockn{L}_i \in \locksin{\Gamma'} }
  {\Gamma, \Gamma' \yields \Gamma'' \admbra{\subkeyoe{t_1}{\lockn{L}_1}, \dots, \subkeyoe{t_n}{\lockn{L}_n}} \tele} \and
  \inferrule*[left=counit,fraction={-{\,-\,}-}]
  {\Gamma, \Gamma'' \yields a : A \and \Gamma, \Gamma' \yields t_i : \Tiny
    \text{ for } \lockn{L}_i \in \locksin{\Gamma'} }
  {\Gamma, \Gamma', \Gamma'' \admbra{\subkeyoe{t_1}{\lockn{L}_1}, \dots, \subkeyoe{t_n}{\lockn{L}_n}} \yields a \admbra{\subkeyoe{t_1}{\lockn{L}_1}, \dots, \subkeyoe{t_n}{\lockn{L}_n}} : A \admbra{\subkeyoe{t_1}{\lockn{L}_1}, \dots, \subkeyoe{t_n}{\lockn{L}_n}} }
\end{mathpar}

Because we will so often be working with a sequence of terms and
locks, we will abbreviate
$\admbra{\subkeyoe{t_1}{\lockn{L}_1}, \dots,
  \subkeyoe{t_n}{\lockn{L}_n}}$ to simply
$\subkey{\vec{t}}{L}$.

Formally, the telescope $\Gamma''$ should be annotated in the syntax
of the operation, as it is not inferable from the raw syntax.

\paragraph{Variable Usage.} To use a variable $x : A$ that is
behind some context locks, we must provide a term $t_i : \Tiny$
for each lock between that variable and the front of the context.

\begin{mathpar}
  \inferrule*[left=var]
  {\Gamma, x : A, \Gamma' \yields \vec{t} : \Tiny
    \text{ for } \lockn{L} \in \locksin{\Gamma'} }
  {\Gamma, x : A, \Gamma' \yields x \substuck{\vec{t}}{L} : A\subkey{\vec{t}}{L}}
\end{mathpar}

The type of the variable usage has the admissible counit rule
applied to it: typically these counits will not be stuck on the type
unless $A$ is itself a variable.

\paragraph{Substitution.} Ordinary substitution into a variable
continues into the keys associated with the variable. When we find
the variable we are substituting for, the `stuck' counits are turned
back into admissible ones, which then proceed into the substituted
term:
\begin{align*}
  \Gamma, \Gamma'[a/x] 
  &\yields (y \substuck{\vec{t}}{L})[a/x] :\defeq y\substuck{\vec{t}[a/x]}{L} \\
  \Gamma, \Gamma'[a/x]
  &\yields (x \substuck{\vec{t}}{L})[a/x] :\defeq a\subkey{\vec{t}[a/x]}{L}
\end{align*}


\paragraph{Computing the Counit.} We will now explain how the
counit operation is actually computed on syntax. We begin with the
simplest possible instance: the addition a single additional key with using a
closed term of $\Tiny$:
\begin{mathpar}
  \inferrule*[left=counit,fraction={-{\,-\,}-}]
  {\Gamma, \Gamma'' \yields b : B \and \cdot \yields t : \Tiny}
  {\Gamma, \ctxlock{L}, \Gamma''\subkey{t}{L}\yields b\subkey{t}{L} : B\subkey{t}{L}}
\end{mathpar}
This is calculated by induction on $b$ until we reach an instance of
the variable rule, say for a variable $x : A$.

As in the definition of single-variable substitution, there are
cases depending on where in the context $x$ lies. The typical case
is when $x : A$ is in $\Gamma$, so the variable usage looks like
\begin{align*}
  \Gamma_1, x : A, \Gamma_2, \Gamma''
  &\yields x \stubra{\subkeyo{\vec{j}}{J}, \subkeyo{\vec{k}}{K}} : A \admbra{\subkeyo{\vec{j}}{J}, \subkeyo{\vec{k}}{K}}
\end{align*}  
where $\lockn{J}$ are the locks in $\Gamma_2$
and $\lockn{K}$ are the locks in $\Gamma''$. The new $\subkey{t}{L}$ is slot into place, and left stuck:
\begin{align*}
  \Gamma_1, x : A, \Gamma_2, \ctxlock{L}, \Gamma''\subkey{t}{L} \yields
  &x \stubra{\subkeyo{\vec{j}}{J}, \subkeyo{\vec{k}}{K}} \subkey{t}{L}
    :\defeq x \stubra{\subkeyo{\vec{j}}{J}, \subkeyo{t}{L}, \subkeyo{\vec{k}}{K}}
\end{align*}

When working informally, this means counting how many
$\ctxlocke{\lockn{K}}$ are created between the application of
$\subkey{t}{L}$ and the variable usage (by the use of some of the
typing rules to be introduced later), and placing $\subkeyo{t}{L}$
to the left of all these freshly created $\ctxlocke{\lockn{K}_i}$.

If instead $x : A$ is in $\Gamma''$, the variable usage looks like
\begin{align*}
  \Gamma, \Gamma_1'', x : A, \Gamma_2''
  &\yields x \stubra{\subkeyo{\vec{k}}{K}} : A\subkey{\vec{k}}{K}
\end{align*}
where $\lockn{K}_{1}, \dots, \lockn{K}_n$ are the locks in
$\Gamma_2''$. Access to $x$ is not affected by the addition of the
lock $\ctxlock{L}$ (which is placed to the left of it), so the
variable usage is left unchanged:
\begin{align*}
  \Gamma, \ctxlock{L}, \Gamma_1''\subkey{t}{L}, x : A\subkey{t}{L}, \Gamma_2''\subkey{t}{L}
  &\yields x\stubra{\subkeyo{\vec{k}}{K}} \subkey{t}{L}
    :\defeq x\stubra{\subkeyo{\vec{k}}{K}} 
\end{align*}

The counit is more complex to evaluate when arbitrary open terms of
$\Tiny$ are involved. Two complications can arise:

\begin{itemize}
\item First, $t : \Tiny$ may not be closed, so that in
  $\Gamma_1, x : A, \Gamma_2, \Gamma' \yields t : \Tiny$, the
  variables used in the term $t$ are now locked behind the fresh locks
  in $\Gamma''$ by the time we reach the variable $x$. We therefore
  apply $\admbra{\subkeyo{\vec{k}}{K}}$ to the term $t$, so:
  \begin{align*}
    x \stubra{\subkeyo{\vec{j}}{J}, \subkeyo{\vec{k}}{K}} \subkey{t}{L}
    :\defeq 
    x \stubra{\subkeyo{\vec{j}}{J}, \subkeyo{t\subkey{\vec{k}}{K}}{L}, \subkeyo{\vec{k}}{K}}
  \end{align*}
\item Instead, the terms $\vec{j}$ and $\vec{k}$ may not be closed,
  and so some of the the variables used in them might now lie behind
  the new lock $\ctxlock{L}$. The term $t$ is used to unlock them:
  \begin{align*}
    x \stubra{\subkeyo{\vec{j}}{J}, \subkeyo{\vec{k}}{K}} \subkey{t}{L}
    :\defeq 
    x \stubra{\subkeyo{\vec{j}\subkey{t}{L}}{J}, \subkeyo{t}{L}, \subkeyo{\vec{k}\subkey{t}{L}}{K}}
  \end{align*}

\item Finally, and slightly horrifyingly, the two complications can happen
  at the same time, yielding the final, general definition:
  \begin{align*}
    x \stubra{\subkeyo{\vec{j}}{J}, \subkeyo{\vec{k}}{K}} \subkey{t}{L}
    :\defeq 
    x \stubra{\subkeyo{\vec{j}\subkey{t}{L}}{J}, \subkeyo{t\subkey{\vec{k}\subkey{t}{L}}{K}}{L}, \subkeyo{\vec{k}\subkey{t}{L}}{K}}
  \end{align*}
  
\end{itemize}

\paragraph{Examples of the Counit.} The fully general rule is the
worst-case scenario; the vast majority of cases encountered in
practice are simpler. First, the counit commutes with ordinary
term-formers such as pairing:
\begin{alignat*}{3}
  &x : \NN,  y : \NN
  &&\yields (x, y) &&: \NN \times \NN \\
  &x : \NN,  y : \NN, \ctxlock{L}, i : \Tiny
  &&\yields (x, y)\subkey{i}{L} \\
  &&&\defeq (x\substuck{i}{L}, y\substuck{i}{L}) &&: \NN \times \NN 
\end{alignat*}
Any variables which are bound in a term do not get annotated with the new counit:
\begin{alignat*}{3}
  &x : \NN
  &&\yields (\lambda y. x + y) &&: \NN \to \NN  \\
  &x : \NN,\ctxlock{L}, i : \Tiny
  &&\yields (\lambda y. x + y)\subkey{i}{L} \\
  &&&\defeq (\lambda y. x\substuck{i}{L} + y) &&: \NN \to \NN 
\end{alignat*}
When the variable $y : \NN$ is bound, it is added to the context
\emph{after} the lock $\ctxlock{L}$, the variable usage therefore
needs no annotation to be well-formed. In particular, the rule has no
effect on closed terms:
\begin{alignat*}{3}
  &\cdot
  &&\yields 2 &&: \NN  \\
  &\ctxlock{L}, i : \Tiny
  &&\yields 2\subkey{i}{L} \\
  &&&\defeq 2 &&: \NN
\end{alignat*}

The counit is also applied to the type of a term, and is computed in a
similar way.
\begin{alignat*}{3}
  &A : \univ, B : A \to \univ, f : \prd{x : A} B(x)
  &&\yields f &&: \prd{x : A} B(x) \\
  &A : \univ, B : A \to \univ, f : \prd{x : A} B(x), \ctxlock{L}, i : \Tiny
  &&\yields f\subkey{i}{L} &&: \left(\prd{x : A} B(x)\right)\subkey{i}{L} \\
  &&&\defeq f\substuck{i}{L} &&: \prd{x : A\subkey{i}{L}} (B(x))\subkey{i}{L} \\
  &&&\defeq f\substuck{i}{L} &&: \prd{x : A\substuck{i}{L}} B\substuck{i}{L}(x)
\end{alignat*}
Because $x$ is bound in the $\Pi$ type-former, its usage is not
annotated in the codomain $B\substuck{i}{L}(x)$. The type-family
$B\substuck{i}{L}$ has type
$(A \to \univ)\substuck{i}{L} \defeq (A\substuck{i}{L} \to \univ)$, so
applying it to $x\substuck{i}{L}$ is well-formed.

When a variable already has a stuck counit attached, the result of
applying an additional counit will depend on the relative position of
the two locks. For the simplest possible case, the variable usage
\[x : A,\ctxlock{K}, k : \Tiny \yields x\substuck{k}{K} : A,\] there are
already a few possibilities.

\begin{itemize}
\item If the new lock is placed at the end of the context, the terms
  used on the existing lock will also end up with the stuck counit on
  their variables.
  \begin{alignat*}{3}
    &x : A,\ctxlock{K}, k : \Tiny
    &&\yields x\substuck{k}{K} &&: A \\
    &x : A,\ctxlock{K}, k : \Tiny, \ctxlock{L}, i : \Tiny
    &&\yields x\substuck{k}{K} \subkey{i}{L} \\
    &&&\defeq x\stubra{\subkeyo{k\subkey{i}{L}}{K}, \subkeyo{i}{L}}  \\
    &&&\defeq x\stubra{\subkeyo{k\substuck{i}{L}}{K}, \subkeyo{i}{L}} &&: A
  \end{alignat*}
\item Moving one step to the left, we may place the new lock in
  between $\ctxlock{K}$ and $k : \Tiny$. In this case neither of the
  variables $i$ or $k$ are locked, and so do not need stuck counit of
  their own to be used.
  \begin{alignat*}{3}
    &x : A, \ctxlock{K}, k : \Tiny
    &&\yields x\substuck{k}{K} &&: A \\
    &x : A, \ctxlock{K}, \ctxlock{L}, i : \Tiny, k : \Tiny
    &&\yields x\substuck{k}{K} \subkey{i}{L} \\
    &&&\defeq x\stubra{\subkeyo{k\subkey{i}{L}}{K}, \subkeyo{i}{L}}  \\
    &&&\defeq x\stubra{\subkeyo{k}{K}, \subkeyo{i}{L}} &&: A
  \end{alignat*}
\item One step further, we may place the new lock \emph{before} an
  existing one, in which case the term $i$ is the one which gains a
  stuck counit in order to be used:
  \begin{alignat*}{3}
    &x : A, \ctxlock{K}, k : \Tiny
    &&\yields x\substuck{k}{K} &&: A \\
    &x : A, \ctxlock{L}, i : \Tiny, \ctxlock{K}, k : \Tiny
    &&\yields x\substuck{k}{K} \subkey{i}{L} \\
    &&&\defeq x\stubra{\subkeyo{i\subkey{k\subkey{i}{L}}{K}}{L}, \subkeyo{k}{K}} \\
    &&&\defeq x\stubra{\subkeyo{i\subkey{k}{K}}{L}, \subkeyo{k}{K}} \\
    &&&\defeq x\stubra{\subkeyo{i\substuck{k}{K}}{L}, \subkeyo{k}{K}} &&: A
  \end{alignat*}
\item Finally, we may place the new lock at the start of the context,
  in which case the variable usage is not changed at all:
  \begin{alignat*}{3}
    &x : A, \ctxlock{K}, k : \Tiny
    &&\yields x\substuck{k}{K} &&: A \\
    &\ctxlock{L}, i : \Tiny, x : A, \ctxlock{K}, k : \Tiny
    &&\yields x\substuck{k}{K} \subkey{i}{L} \\
    &&&\defeq x\substuck{k}{K} &&: A
  \end{alignat*}
\end{itemize}

\paragraph{Unit.} To complete the implementation of the adjunction,
we have an admissible rule representing the unit map. The following
rule is precomposition with the substitution
$\Gamma \to (\Gamma,i : \Tiny, \ctxlock{L})$, suitably generalised
with a telescope $\Gamma'$:
\begin{mathpar}
  \inferrule*[left=unit-tele,fraction={-{\,-\,}-}]
  {\Gamma, i : \Tiny, \ctxlock{L} \yields \Gamma' \tele}
  {\Gamma \yields \Gamma'\sublock{i}{L} \tele} \and
  \inferrule*[left=unit,fraction={-{\,-\,}-}]
  {\Gamma, i : \Tiny, \ctxlock{L}, \Gamma' \yields a : A}
  {\Gamma, \Gamma'\sublock{i}{L} \yields a\sublock{i}{L} : A\sublock{i}{L}}
\end{mathpar}
As with the $\subkey{t}{L}$ operation, we use substitution-like
syntax for the unit map because, like substitution, it commutes past
everything until it reaches a variable. The $\rbind{i}$ notation is
chosen to match the syntax that is used for the type-former later.

The lock and key then `click together', yielding an ordinary
substitution. In the simplest case, when $t : \Tiny$ is closed and
$x$ has a single matching stuck \rulen{counit}.
\begin{align*}
  x \substuck{t}{L}\sublock{i}{L} :\defeq x[t/i]
\end{align*}
If $t$ is an open term, the unit rule must continue into $t$:
\begin{align*}
  x \substuck{t}{L}\sublock{i}{L} :\defeq x[t\sublock{i}{L}/i]
\end{align*}

Of course, because $x$ is already a variable there are only two
possibilities: either it is equal to the variable $i$ bound in the
unit rule or it is not.
\begin{align*}
  i \substuck{t}{L}\sublock{i}{L} &:\defeq t\sublock{i}{L} \\
  x \substuck{t}{L}\sublock{i}{L} &:\defeq x \qquad \text{otherwise}
\end{align*}

If $x$ has several stuck counits, the operation needs to continue
into the other stuck counits.
\begin{align*}
  i \stubra{\subkeyo{\vec{j}}{J}, \subkeyo{t}{L}, \subkeyo{\vec{k}}{K}}\sublock{i}{L}
  &:\defeq {} t\sublock{i}{L} \\
  x \stubra{\subkeyo{\vec{j}}{J}, \subkeyo{t}{L}, \subkeyo{\vec{k}}{K}}\sublock{i}{L} 
  &:\defeq {} x \stubra{\subkeyo{\vec{j}\sublock{i}{L}}{J}, \subkeyo{\vec{k}\sublock{i}{L}}{K}}
\end{align*}

In contrast to the counit, this rule is never stuck. After any
application of $\sublock{i}{L}$, the lock $\lockn{L}$ and the
variable $i$ no longer appear in the resulting term
$a\sublock{i}{L}$.
  
\paragraph{Examples of the Unit.}
Now, applying the unit rule. Similar to the counit, the unit commutes
with ordinary type formers (in these cases having no effect, because
the variables used come after the lock).
\begin{alignat*}{3}
  &i : \Tiny, \ctxlock{L},x : \NN,  y : \NN
  &&\yields (x, y) &&: \NN \times \NN \\
  &x : \NN,  y : \NN
  &&\yields (x, y)\sublock{i}{L} \\
  &&&\defeq (x\sublock{i}{L}, y\sublock{i}{L})\\
  &&&\defeq (x, y) &&: \NN \times \NN  \\
  &i : \Tiny, \ctxlock{L}, x : \NN
  &&\yields (\lambda y. x + y) &&: \NN \to \NN  \\
  &x : \NN
  &&\yields (\lambda y. x + y)\sublock{i}{L} \\
  &&&\defeq (\lambda y. x\sublock{i}{L} + y) \\
  &&&\defeq (\lambda y. x + y) &&: \NN \to \NN
\end{alignat*}
If the variable used is not the one that the unit is being applied to
then it remains unchanged:
\begin{alignat*}{2}
  &x : A, i : \Tiny,\ctxlock{L}, k : \Tiny
  &&\yields x\substuck{k}{L} : A \\
  &x : A, k : \Tiny
  &&\yields (x\substuck{k}{L})\sublock{i}{L} \\
  &&&\defeq x[k/i] \\
  &&&\defeq x : A
\end{alignat*}
The simplest interesting case is a variable usage that matches the
variable that the unit is being applied to.
\begin{alignat*}{2}
  &i : \Tiny,\ctxlock{L}, k : \Tiny
  &&\yields i\substuck{k}{L} : \Tiny \\
  &k : \Tiny
  &&\yields (i\substuck{k}{L})\sublock{i}{L} \\
  &&&\defeq i[k\sublock{i}{L}/i] \\
  &&&\defeq k\sublock{i}{L} \\
  &&&\defeq k : \Tiny 
\end{alignat*}
That is not the only way we could have gained access to $i$. We can
use the term of $\Tiny$ just constructed to gain access to $i$ a
second time:
\begin{alignat*}{2}
  &i : \Tiny, \ctxlock{L}, k : \Tiny
  &&\yields i\substuck{i\substuck{k}{L}}{L} : \Tiny \\
  &k : \Tiny
  &&\yields \left( i\substuck{i\substuck{k}{L}}{L} \right)\sublock{i}{L} \\
  &&&\defeq i\substuck{k}{L}\sublock{i}{L} \\
  &&&\defeq k\sublock{i}{L} \\
  &&&\defeq k : \Tiny
\end{alignat*}
Nothing stops us from iterating this forever:
\begin{alignat*}{2}
  &i : \Tiny, \ctxlock{L}, k : \Tiny
  &&\yields
     i\substuck{i\substuck{i\substuck{k}{L}}{L}}{L} : \Tiny
\end{alignat*}
Semantically, the context corresponds to $i : \Tiny \to \Tiny$ and
$k : \Tiny$, and the term is an incarnation of the iterated
application $i(i(\dots i(k)))$. The unit $\sublock{i}{L}$ replaces $i$
with the identity, and so the entire term reduces to $k$ regardless of
how many iterations we do.

The result of applying the counit may be different for each variable,
depending on what terms are used in the attached counit:
\begin{alignat*}{2}
  &i : \Tiny,\ctxlock{L}, j : \Tiny, k : \Tiny
  &&\yields (i\substuck{j}{L}, i\substuck{k}{L}) : \Tiny \times \Tiny \\
  &j : \Tiny, k : \Tiny
  &&\yields (i\substuck{j}{L}, i\substuck{k}{L})\sublock{i}{L} \\
  &&&\defeq (i\substuck{j}{L}\sublock{i}{L}, i\substuck{k}{L}\sublock{i}{L}) \\
  &&&\defeq (i[j/i], i[k/i]) \\
  &&&\defeq (j, k) : \Tiny \times \Tiny
\end{alignat*}

At this point it may appear as though the counits are just delayed
substitutions that are `activated' by the unit, but this is not the
case. Suppose we have global elements $1, 2, 3, 4 : \Tiny$, and
consider a term
\begin{align*}
  x : \Tiny, y : \Tiny, \ctxlock{L}, \ctxlock{K}
  &\yields
    (x\stubra{\subkeyo{1}{L},\subkeyo{2}{K}},y\stubra{\subkeyo{3}{L},\subkeyo{4}{K}})
    : \Tiny \times \Tiny
\end{align*}
Applying two different substitutions allows us to select which of the
stuck keys to apply:
\begin{align*}
  &\phantom{{}\defeq{}} (x\stubra{\subkeyo{1}{L},\subkeyo{2}{K}}, y\stubra{\subkeyo{3}{L},\subkeyo{4}{K}})[i/x][j/y][\rbind{i}/\ctxlock{L}][\rbind{j}/\ctxlock{K}] \\
  &\defeq (i\stubra{\subkeyo{1}{L},\subkeyo{2}{K}}, j\stubra{\subkeyo{3}{L},\subkeyo{4}{K}})[\rbind{i}/\ctxlock{L}][\rbind{j}/\ctxlock{K}] \\
  &\defeq (i\stubra{\subkeyo{1}{L},\subkeyo{2}{K}}[\rbind{i}/\ctxlock{L}][\rbind{j}/\ctxlock{K}], j\stubra{\subkeyo{3}{L},\subkeyo{4}{K}}[\rbind{i}/\ctxlock{L}][\rbind{j}/\ctxlock{K}]) \\
  &\defeq (1, 4)
\end{align*}
and
\begin{align*}
  &\phantom{{}\defeq{}} (x\stubra{\subkeyo{1}{L},\subkeyo{2}{K}}, y\stubra{\subkeyo{3}{L},\subkeyo{4}{K}})[j/x][i/y][\rbind{i}/\ctxlock{L}][\rbind{j}/\ctxlock{K}] \\
  &\defeq (j\stubra{\subkeyo{1}{L},\subkeyo{2}{K}}, i\stubra{\subkeyo{3}{L},\subkeyo{4}{K}})[\rbind{i}/\ctxlock{L}][\rbind{j}/\ctxlock{K}] \\
  &\defeq (j\stubra{\subkeyo{1}{L},\subkeyo{2}{K}}[\rbind{i}/\ctxlock{L}][\rbind{j}/\ctxlock{K}], i\stubra{\subkeyo{3}{L},\subkeyo{4}{K}}[\rbind{i}/\ctxlock{L}][\rbind{j}/\ctxlock{K}]) \\
  &\defeq (2, 3)
\end{align*}
We will be able to cause these different substitutions to occur once
we have the type-former. So the \emph{user} of a term containing stuck
keys gets to choose which of the terms is eventually used, but doesn't
get to choose the actual term that gets plugged in. And the stored
terms for the other locks are completely lost!

These rules are peculiar, but if we insist on pushing the admissible
rules to the leaves it seems something like this is forced on us by
the setting we are trying to capture.

\section{The Amazing Right Adjoint}\label{sec:type}

The type-former is made right adjoint to $\ctxlock{L}$ in the
`obvious' way, following the pattern of~\cite{clouston:fitch-style,
  drats, fitchtt}.
\begin{mathpar}
  \inferrule*[left=$\rformsym$-form]
  {\Gamma, \ctxlock{L} \yields A : \univ}
  {\Gamma \yields \rform{L} A : \univ} \and
  \inferrule*[left=$\rformsym$-intro]
  {\Gamma, \ctxlock{L} \yields a : A}
  {\Gamma \yields \rintro{L} a : \rform{L} A} \and
  \inferrule*[left=$\rformsym$-elim]
  {\Gamma, i : \Tiny \yields r : \rform{L} A}
  {\Gamma \yields r(\rbind{i}) : A\sublock{i}{L}} \\
  (\rintro{L}{a})(\rbind{i}) \defeq a\sublock{i}{L} \and
  r \defeq \rintro{L}{(r\subkey{i}{L}(\rbind{i}))}
\end{mathpar}

In words:
\begin{itemize}
\item \rulen{$\rformsym$-form} and \rulen{$\rformsym$-intro}: For any
  term $a : A$, that uses all ambient variables locked behind
  $\ctxlock{L}$, we can bind the lock $\ctxlock{L}$ and form a term
  $\rintro{L} a : \rform{L} A$. The term syntax is intended to be
  reminiscent of an ordinary $\lambda$-abstraction.

  Reading upwards, going under $\rform{L}$ or $\rintro{L} -$ means
  that all extant variables become locked by $\ctxlock{L}$, so every
  future use of those variables in $a$ must have an attached
  $-\substuck{t}{L}$ to be well formed.

  These are precisely the formation and introduction rule for a
  dependent right adjoint type~\cite{drats}.

  If $A$ is a closed type, we will just write $\rformu A$ rather than
  binding an unused lock name.
  
\item \rulen{$\rformsym$-elim}: If using an additional assumption
  $i : \Tiny$ we can produce a term $r : \rform{L} A$, then we can
  amazingly apply $r$ to the fresh $i$ to form
  $r(\rbind{i}) : A\sublock{i}{L}$. We are free to ignore the new
  assumption $i$ in the term $r$ if we wish, just as a constant
  function may ignore its argument.

  The syntax is, put mildly, unorthodox: to keep the analogy with
  ordinary function application, the variable $i$ is bound in the
  parentheses on the right, but then is \emph{in scope in the body of
    the `function' to the left}\footnote{If we took Lawvere
    seriously~\cite{lawvere:adjointness-foundations} and wrote
    function application as $xf$, then we would not have this
    problem!}.

\item \rulen{$\rformsym$-beta}: Binding a lock $\rintro{L} a$ and
  amazingly applying the result to $\rbind{i}$ reduces to
  $a\sublock{i}{L}$, with
  $(\rintro{L}{a})(\rbind{i}) \defeq a\sublock{i}{L}$ in perfect
  analogy with $(\lambda x. b)(a) \defeq b[a/x]$.
  
\item \rulen{$\rformsym$-eta}: Any term $r : \rform{L} A$ is equal to
  the term that binds a lock and immediately amazingly applies $r$ to
  it, in analogy with $f \defeq (\lambda x. f(x))$.
\end{itemize}

\subsection{Basic Examples}

To begin, some maps that are easily definable from the rules. The
functor $\Tiny \to -$ is a monad, and its right adjoint $\rformu{}$ is
therefore a comonad.

\begin{definition}
  If $A$ is (for now) a closed type, define an 
  extract map $\extract_A : \rformu{A} \to A$ by
  \begin{align*}
    \extract_A &: \rformu{A} \to A \\
    \extract_A(r) &:\defeq r(\rbind{i})
  \end{align*}
  In words, to get a term of $A$ we need a term of $\rformu{A}$ that
  we can amazingly evaluate on an additional assumption $i : \Tiny$.
  But we already have the variable $r : \rformu{A}$, so we have no
  need to use the additional assumption.
\end{definition}

We have no way of quantifying over \emph{closed} types internally, so
any definition mentioning closed types should be considered
provisional: a corresponding fully internal definition will be given
shortly after.

\begin{definition}
  If $f : A \to B$ is a (for now) closed function, define $\rformu f$ by
  \begin{align*}
    \rformu f &: \rformu A \to \rformu B \\
    (\rformu f)(r) &:\defeq \rintro{L}{f(r\substuck{i}{L}(\rbind{i}))}
  \end{align*}
\end{definition}

This should be compared with the definition of function
post-composition, as the shape of the definition is the same:
\begin{align*}
  f \circ - &: (C \to A) \to (C \to B) \\
  (f \circ -)(r) &:\defeq \lambda c. f(r(c))
\end{align*}

In words, we start with $r : \rformu A$. To produce a term of
$\rformu B$ we have to produce term of $B$, but with all our present
assumptions locked behind $\ctxlock{L}$. There is a function
$f : A \to B$ available, so we just need access to a term of $A$. We
cannot amazingly evaluate $r : \rformu A$ to get this term of $A$,
because $r$ has become trapped behind $\ctxlock{L}$. But we \emph{can}
amazingly evaluate $r\substuck{i}{L} : \rformu A$ on an additional
assumption $i : \Tiny$. This gets us an argument of type $A$ that we
can apply $f$ to. This time, the presence of the additional assumption
$i : \Tiny$ is crucial: it is exactly what we needed to gain access
to $r$.

If $A$ is not a closed type, then the type of
$\extract_A : \rformu{A} \to A$ is not well-formed: once we go
under $\rformsym$, we lose access to the variable $A$. Similarly, if
$f$ is not a closed function, then we cannot produce
$\rformu f : \rformu A \to \rformu B$, even if the types $A$ and $B$
are themselves closed, because the function $f$ (as a variable) also
becomes locked behind $\ctxlock{L}$ when the argument $r$ does.

However: if \emph{all} the inputs to the above constructions are
provided under $\rformu$, then we can have access to them where we
need it. Suppose instead that $A : \rformu \univ$. Then, we may use
$A$ as an input to $\rformsym$ by forming
$\rform{L}{A\substuck{i}{L}(\rbind i)} : \univ$. Our map $\extract_A$
can now be correctly typed:

\begin{definition}
  For $A : \rformu \univ$, define a map $\extract_A$
  \begin{align*}
    &\extract_A : \rform{L}{A\substuck{i}{L}(\rbind i)} \to A(\rbind i) \\
    &\extract_A(r) :\defeq r(\rbind{i})
  \end{align*}  
\end{definition}

And similarly for functoriality:
\begin{definition}
  For $A,B : \rformu \univ$ and
  $f : \rform{L}\left((\relim{i.A\substuck{i}{L}}) \to (\relim{i.B\substuck{i}{L}})\right)$, define $\rformu f$ by
  \begin{align*}
    \rformu f &: \rform{L}\left(\relim{i.A\substuck{i}{L}}\right) \to \rform{L}\left( \relim{i.B\substuck{i}{L}}\right) \\
    \rformu f(r) &:\defeq \rintro{L}{(\relim{j. f\substuck{j}{L}})(\relim{i.r\substuck{i}{L}})}
  \end{align*}
\end{definition}

This simple example is already very noisy, so to cut down on symbols
we introduce two pieces of syntactic sugar.

\paragraph{First:} Like ordinary \rulen{$\to$-intro}, and as with any
negative type-former, the \rulen{$\rformsym$-intro} rule is
right-invertible and so we will define terms of it via copattern
matching. For example, we will write
\begin{align*}
  r(\ctxlock{L}, \ctxlock{K}) :\defeq \judge
  \quad \text{to mean} \quad
  r :\defeq \rintro{L}{\rintro{K}{\judge}}
\end{align*}
This integrates nicely with ordinary copattern-style definitions of
functions: we write
\begin{align*}
  f(x,y,\ctxlock{L},z) :\defeq \judge
  \quad \text{to mean} \quad
  f :\defeq \lambda x. \lambda y. \rintro{L}{\lambda z. \judge}
\end{align*}
This syntax makes it easy to see at a glance which locks affect each
variable in the term $\judge$: the $x$ and $y$ here are to the left of
$\ctxlock{L}$ in the list of arguments and so are locked by
$\ctxlock{L}$, whereas $z$ is not locked at all.

Similarly, we might evaluate some combination of $\to$ and $\rformsym$
by applying it to a tuple of arguments, as in $f(a, b, \rbind{i}, c)$.
The `amazing' argument to a lock is a new bound variable $i$ as in
\rulen{$\rformsym$-elim}, but remember that this variable is now in
scope to the left of that argument in the tuple \emph{and} in the
function itself. This form of evaluation computes in the `obvious' way,
with a $\beta$-reduction for each entry in the list of arguments. If
$f$ is defined via the pattern match above, we find
\begin{align*}
  f(a(i), b(i), \rbind{i}, c)
   &\defeq f(a(i))(b(i))(\rbind{i})(c) \\
   &\defeq \left( \lambda x. \lambda y. \rintro{L}{\lambda z. \judge} \right)(a(i))(b(i))(\rbind{i})(c) \\
   &\defeq \left( \rintro{L}{\lambda z. \judge[a(i)/x][b(i)/y]} \right)(\rbind{i})(c) \\
   &\defeq \left( \lambda z. \judge[a(i)/x][b(i)/y]\sublock{i}{L} \right)(c) \\
   &\defeq \judge[a(i)/x][b(i)/y]\sublock{i}{L}[c/z]
\end{align*}
We have arranged the syntax of our admissible rules so that function
definitions and uses click together, as in ordinary function
evaluation: comparing $f(a(i), b(i), \rbind{i}, c)$ and
$f(x,y,\ctxlock{L},z)$, the first arguments become substitutions
$[a(i)/x][b(i)/y]$ and the amazing $\rbind{i}$ `argument' becomes a
`substitution' $\sublock{i}{L}$.

Any ordinary function $f : A \to B$ is tautologically equal to the definition
\begin{align*}
  f(x) :\defeq f(x)
\end{align*}
by $\eta$-expansion. For $\rformsym$, this $\eta$-expansion introduces
an instance of the counit. That is, any term $r : \rform{L}{A}$ is
tautologically equal to the definition
\begin{align*}
  r(\ctxlock{L}) :\defeq r\subkey{i}{L}(\rbind{i})
\end{align*}

\begin{remark}
  Using copattern syntax could also apply to other right adjoint
  modalities, like $\sharp$~\cite{mike:real-cohesive-hott}. In the
  original syntax, functoriality of $\sharp$ on a function
  $f : A \to B$ is defined by
  $\sharp f(u) :\defeq f(u_\sharp)^\sharp$. In a copattern-style
  syntax, this might be written
  $\sharp f(u, \sharp) :\defeq f(u(\sharp))$, or
  $\sharp f(u, \sharp_{\mathsf{i}}) :\defeq f(u(\sharp_{\mathsf{e}}))$
  if we wish to distinguish the symbols for introduction and
  elimination. Here, the interpretation is that a variable placed to
  the left of $\sharp_{\mathsf{i}}$ becomes crisp in the body of the function, and
  an `application' $s(\sharp_{\mathsf{e}})$ is only well-typed when $s$ is crisp.
  In this case, $u$ becomes crisp and so it is valid to apply it to
  $\sharp_{\mathsf{e}}$.
\end{remark}

\paragraph{Second:} We will often need to extract terms of $A$ from
assumptions $r : \rformu{A}$ using the extract map $\extract_A$
defined above. When such assumptions are behind many context locks,
the \rulen{$\rformsym$-elim} rule provides a term $i : \Tiny$ that can
be used to unlock all these locks at once, giving us access to $A$
regardless of where $r$ lies in the context.
\begin{definition}
  If a variable $x : \rform{L} A$ is behind several locks
  $\ctxlocke{\lockn{K}_1} \dots \ctxlocke{\lockn{K}_n}$, then let
  \begin{align*}
    \rget{x} :\defeq (x \stubra{\subkeyoe{i}{\lockn{K}_1},
    \dots, \subkeyoe{i}{\lockn{K}_n}})(\rbind{i}) : A \admbra{\subkeyoe{i}{\lockn{K}_1},
    \dots, \subkeyoe{i}{\lockn{K}_n}}\sublock{i}{L}
  \end{align*}
  in particular, if $x : \rform{L} A$ is behind no locks in the
  context then $\rget{x} :\defeq \extract_A(x)$.
\end{definition}
This notation $\rget{x}$ is intended to evoke the elimination rule
$x_\sharp$ used in cohesive HoTT (and similar). It cannot be
internalised as a function, it is just an admissible piece of syntax.

Using these shorthands, the definition of functoriality is much less
noisy: For $A,B : \rformu \univ$ and
$f : \rform{L}\left(\rget{A} \to \rget{B}\right)$ we have
\begin{align*}
  &\rformu f : \rform{L}\rget{A} \to \rform{L}\rget{B} \\
  &\rformu f(r, \ctxlock{L}) :\defeq \rget{f}(\rget{r})
\end{align*}
The comultiplication of $\rformsym$ can be also written nicely using
this notation. For a type $A : \rformu \univ$, define
\begin{align*}
  &\delta_A : \rformu{\rget{A}} \to \rformu{\rformu{\rget{A}}} \\
  &\delta_A(r, \ctxlock{L}, \ctxlock{K}) :\defeq \rget{r}
\end{align*}

We can show that, as a right adjoint, $\rform{}$ preserves $1$ and pullbacks.
\begin{proposition}
  $\rformu{1} \equiv 1$
\end{proposition}
\begin{proof}
  Any term $r : \rformu{1}$ will $\eta$-expand as
  $r \defeq \rintro{L}{r\substuck{i}{L}(\rbind{i})} \defeq
  \rintro{L}{\star}$ and so $\rformu{1}$ is contractible.
\end{proof}

\begin{proposition}\label{prop:root-preserve-sum}
  If $A : \rformu \univ$ and
  $B : \rformu (\rget{A} \to \univ)$ there is an equivalence
  \begin{align*}
    \rformu \left( \sm{x : \rget{A}}\rget{B}(x) \right) \equiv \sm{x : \rformu \rget{A}} \rformu (\rget{B}(\rget{x}))
  \end{align*}
\end{proposition}
\begin{proof}
  Define maps either way by
  \begin{align*}
    r &\mapsto (\rintro{L}{\proj_1 \rget{r}}, \rintro{K}{\proj_2\rget{r}}) \\
    (s,t) &\mapsto \rintro{L}(\rget{s}, \rget{t})
  \end{align*}
\end{proof}

\begin{proposition}
  For any $A : \rformu{\univ}$ and $r, s : \rformu{\rget{A}}$, there
  is an equivalence
  \begin{align*}
    (r = s) \equiv \rformu \left( \rget{r} = \rget{s} \right)
  \end{align*}
\end{proposition}
\begin{proof}
  By the fundamental theorem of $\Idsym$-types, we just have to check
  that for fixed $r : \rformu{\rget{A}}$, the type
  \begin{align*}
    \sm{s : \rformu{\rget{A}}} \rformu \left( \rget{r} = \rget{s} \right)
  \end{align*}
  is contractible. By \cref{prop:root-preserve-sum}, this
  type is equivalent to
  \begin{align*}
    \rformu{\left(\sm{x : \rget{A}} \rget{r} = x \right)}
  \end{align*}
  and the interior is a contractible pair.
\end{proof}
This is more recognisable as a statement of left-exactness with the
left side $\eta$-expanded:
\begin{corollary}
  $(\rintro{L}{\rget{r}} = \rintro{L}{\rget{s}}) \equiv \rform{L}
  \left( \rget{r} = \rget{s} \right)$
\end{corollary}

\subsection{Adjointness}\label{sec:adjointness}

We now investigate how $\rformu$ relates to $(\Tiny \to -)$. It is
well-known that, although $(\Tiny \to -)$ externally has a right
adjoint, it is not typically a \emph{fibred} right adjoint unless
$\Tiny \equiv 1$~\cite{yetter:tiny, kock:commutation}.

There is, however, a useful internal statement. To begin, unit and
counit maps for the adjunction.
\begin{definition}
  If $A, B : \univ$ are (for now) closed types, there are maps
  \begin{align*}
    \eta_A &: A \to \rformu(\Tiny \to A) \\
    \varepsilon_B &: (\Tiny \to \rformu B) \to B
  \end{align*}
  given by
  \begin{align*}
    \eta_A(a, \ctxlock{L}, t) &:\defeq a \substuck{t}{L} \\
    \varepsilon_B(f) &:\defeq f(i, \rbind{i})
  \end{align*}
  Allowing the maximum amount of dependency, for \emph{any}
  $A : \univ$, there is a map
  \begin{align*}
    \eta_A &: A \to \rform{L} \prd{t : \Tiny} A \substuck{t}{L}
  \end{align*}
  and for \emph{any} $B : \Tiny \to \rformu \univ$ a map
  \begin{align*}
    \varepsilon_B &: \left(\prd{t:\Tiny} \rform{K} \rget{B(t)} \right) \to B(j, \rbind{j})                    
  \end{align*}
  with the same definitions as in the closed case.
\end{definition}

Besides $\varepsilon_B$ not being definable for an arbitrary type $B$,
we have also seen that $\rformu -$ is not functorial on arbitrary maps
$A \to B$. So to what extent is $(\Tiny \to -) \dashv \rformu -$? We
can prove an internal adjointness statement analogous to one
in~\cite[\S 3]{fhm:left-and-right} for adjunctions in which the
left adjoint is strong monoidal.

\begin{proposition}
  If $A$ and $B$ are closed types then there is an equivalence
  \begin{align*}
    \Phi : \rformu ((\Tiny \to A) \to B) \to (A \to \rformu B)
  \end{align*}
\end{proposition}
Happily, we can allow $A$ and $B$ to be dependent types:
\begin{proposition}\label{prop:adj}
  For any $A : \univ$ and $B : \rformu \univ$, there is an equivalence
  \begin{align*}
    \rform{L} ((\prd{t:\Tiny} A\substuck{t}{L}) \to \rget{B}) \to (A \to \rform{L} \rget{B})
  \end{align*}
\end{proposition}
And in fact, we can also let $B$ depend on $A$. In full generality:
\begin{proposition}\label{prop:dep-adj}
  For any $A : \univ$ and $B : A \to \rformu \univ$, there is an
  equivalence
  \begin{align*}
    \Phi : \rform{K} \left(\prd{h:\prd{t:\Tiny} A\substuck{t}{L}} B\substuck{i}{L}(h(i), \rbind{i})\right) \to \prd{a:A} \rform{L} \rget{B(a)}
  \end{align*}
  given by
  \begin{alignat*}{2}
    \Phi(r, a, \ctxlock{L}) &:\defeq r\substuck{i}{L}(\rbind{i}, \lambda t. a \substuck{t}{L}) \\
    \Phi^{-1}(f, \ctxlock{K}, h) &:\defeq f\substuck{k}{K}(h(k), \rbind{k})
  \end{alignat*}
\end{proposition}
\begin{proof}
  All types involved have definitional $\eta$-principles, so this is a
  matter of crunching through the composites in both directions.

  First, we $\eta$-expand $\Phi^{-1}(\Phi(r))$ to $\rintro{J} \lambda h. \Phi^{-1}(\Phi(r\substuck{j}{J}), \rbind{j}, h)$ and simplify the body:
  \begin{align*}
    &\Phi^{-1}(\Phi(r\substuck{j}{J}), \rbind{j}, h) \\
    &\defeq \Phi(r\substuck{j}{J})\subkey{k}{K}\sublock{j}{K}(h(k), \rbind k)
    && \text{(Definition of $\Phi^{-1}$)} \\
    &\defeq \Phi(r\substuck{k}{J})(h(k), \rbind k)
    && \text{(Computing $\sublock{j}{K}$)} \\
    &\defeq r\stubra{\subkeyo{k}{K},\subkeyo{i}{L}}(\rbind{i}, \lambda t. h(k) \substuck{t}{L}) \sublock{k}{L}
    && \text{(Definition of $\Phi$)} \\
    &\defeq r\substuck{i}{K}(\rbind{i}, \lambda t. h(t))
    && \text{(Computing $\sublock{k}{L}$)} \\      
    &\defeq r\substuck{i}{K}(\rbind{i}, h)
    && \text{($\eta$ for functions)}
  \end{align*}
  which is precisely the $\eta$-expansion of $r$.

  The other way, $\Phi(\Phi^{-1}(f))$ expands to $\lambda a. \rintro{J} \Phi (\Phi^{-1}(f\substuck{j}{J}, a, \rbind{j}))$ and:
  \begin{align*}
    &\Phi (\Phi^{-1}(f\substuck{j}{J}), a\substuck{j}{J}, \rbind{j}) \\
    &\defeq \Phi^{-1}(f\substuck{j}{J})\substuck{i}{L}(\rbind{i}, \lambda t. a\substuck{j}{J} \substuck{t}{L})\sublock{j}{L}
    && \text{(Definition of $\Phi$)} \\
    &\defeq \Phi^{-1}(f\substuck{i}{J})(\rbind{i}, \lambda t. a\substuck{t}{J})
    && \text{(Computing $\sublock{j}{L}$)} \\
    &\defeq f\stubra{\subkeyo{i}{J}, \subkeyo{k}{K}}\sublock{i}{K}((\lambda t. a \substuck{t}{J})(k), \rbind{k})
    && \text{(Definition of $\Phi^{-1}$)} \\
    &\defeq f\substuck{k}{J}((\lambda t. a \substuck{t}{J})(k), \rbind{k})
    && \text{(Computing $\sublock{i}{K}$)} \\
    &\defeq f\substuck{k}{J}(a \substuck{k}{J}, \rbind{k})
    && \text{($\beta$ for functions)}
  \end{align*}
  is the $\eta$-expansion of $f$.

\end{proof}

\begin{remark}
  Internal equivalences of this kind for an arbitrary adjoint modality
  have been defined in MTT, where the modality is instead a positive
  type-former, see~\cite[Proposition 3.4]{transpension}
  and~\cite[Section 10.4]{mtt}. The fact that this internal
  formulation circumvents the `no-go' theorem of~\cite{lops} was also
  noted in~\cite[Section 10.1]{transpension}.
\end{remark}

\section{Constructions}\label{sec:constructions}

\subsection{Higher-Dimensional Induction}\label{sec:induction}
Since we have given $\Tiny \to -$ a right adjoint, it now preserves
colimits. In particular, we support the same kind of
``higher-dimensional pattern matching'' on functions out of $\Tiny$,
as suggested in~\cite[\S 2.4]{transpension}. In this section we derive
an induction principle for functions of $\Tiny$ into coproducts,
thought we expect easy generalisations to other inductive types.

\begin{proposition}\label{prop:higher-match}
  Suppose types $A, B : \univ$ and a type family \[P : \rform{L}\left( \left( \prd{t : \Tiny} A\substuck{t}{L} + B\substuck{t}{L} \right) \to \univ\right).\]
  Then there is a map 
  \begin{align*}
    \mathsf{hind} : &\,\rform{L} \left(\prd{l : \prd{t : \Tiny} A\substuck{t}{L}} \rget{P}(\inl \circ l) \right) \\
    \times &\,\rform{L} \left(\prd{r : \prd{t : \Tiny} B\substuck{t}{L}} \rget{P}(\inr \circ r) \right) \\
    \to &\,\rform{L} \left(\prd{f : \prd{t : \Tiny} A\substuck{t}{L} + B\substuck{t}{L}} \rget{P}(f) \right)
  \end{align*}
  such that $\mathsf{hind}(g, h)$ computes via
  \begin{alignat*}{2}
    &\mathsf{hind}(g, h)(\ctxlock{L}, \lambda t. \inl(\rget{l}(t))) &&\defeq \rget{g}(\rget{l}) \\
    &\mathsf{hind}(g, h)(\ctxlock{L}, \lambda t. \inr(\rget{r}(t))) &&\defeq \rget{h}(\rget{r})
  \end{alignat*}
  for any $l : \rform{L}(\prd{t:T}A\substuck{t}{L})$ and $r : \rform{L}(\prd{t:T} B\substuck{t}{L})$.
\end{proposition}
If we like, we can always extract the underlying function
$\mathsf{hind}(g, h)(\rbind i) : \prd{f:\Tiny \to A + B} \rget{P}(f)$ of
the result.
\begin{proof}
  There is a chain of equivalences
  \begin{align*}
    &\rform{L} \left(\prd{l : \prd{t : \Tiny} A\substuck{t}{L}} \rget{P}(\inl \circ l) \right) \times \rform{L} \left(\prd{r : \prd{t : \Tiny} B\substuck{t}{L}} \rget{P}(\inr \circ r) \right) \\
    &\quad \text{(Dependent adjointness)} \\
    &\equiv \left(\prd{a : A} \rform{L} \rget{P}(\lambda i. \inl(a\substuck{i}{L})) \right) \times \left(\prd{b : B} \rform{L} \rget{P}(\lambda i. \inr(b\substuck{i}{L})) \right) \\
    &\quad \text{(Universal property of $+$)} \\
    &\equiv \prd{s : A + B} \rform{L} \rget{P}(\lambda i. s\substuck{i}{L}) \\
    &\quad \text{(Dependent adjointness)} \\
    &\equiv \rform{L} \left(\prd{f : \prd{t : \Tiny} A\substuck{t}{L} + B\substuck{t}{L}} \rget{P}(f) \right)
  \end{align*}
  As a term, the forward map is 
  \begin{align*}
    \mathsf{hind}(g, h)(\ctxlock{L}, f) :\defeq \mathsf{case}_+(f(i), &a. \rintro{K} \rget{g}(\lambda t. a\substuck{t}{K}), \\
    &b. \rintro{K} \rget{h}(\lambda t. b\substuck{t}{K}) )(\rbind i)
  \end{align*}
  We can show directly that this map has the desired computational properties:
  \begin{align*}
    &\mathsf{hind}(g, h)(\ctxlock{L}, \lambda t. \inl(\rget{l}(t))) \\
    &\defeq \mathsf{case}_+(\inl(\rget{l}(i)), a. \rintro{K} \rget{g}(\lambda t. a\substuck{t}{K}), b. \rintro{K} \rget{h}(\lambda t. b\substuck{t}{K}) )(\rbind i) \\
    &\defeq (\rintro{K} \rget{g}(\lambda t. \rget{l}\substuck{t}{K}(i\substuck{t}{K})))(\rbind i) \\
    &\defeq \rget{g}(\lambda t. \rget{l}\substuck{t}{K}(i\substuck{t}{K}))\sublock{i}{K} \\
    &\defeq \rget{g}(\lambda t. \rget{l}(t)) \\
    &\defeq \rget{g}(\rget{l})
  \end{align*}
  and similarly when evaluated at $\lambda t. \inr(\rget{r}(t))$.
\end{proof}
The upshot of the above is that we support higher-dimension pattern
matching of the following form: to define a function
$p : \rform{L} \left(\prd{f : \prd{t : \Tiny} A\substuck{t}{L} +
    B\substuck{t}{L}} \rget{P}(f) \right)$, it suffices to give cases
\begin{alignat*}{2}
  &p(\ctxlock{L}, \lambda t. \inl(\rget{l}(t))) &&:\defeq \dots \\
  &p(\ctxlock{L}, \lambda t. \inr(\rget{r}(t))) &&:\defeq \dots 
\end{alignat*}

An immediate application is:
\begin{proposition}
  For any types $A, B : \univ$, the map
  \begin{align*}
    \mathsf{split}_{A + B} : \left((\Tiny \to A) + (\Tiny \to B)\right) \to (\Tiny \to A + B)
  \end{align*}
  defined by case-splitting is an equivalence.  
\end{proposition}
\begin{proof}
  The inverse map is definable using this new pattern matching, with
  motive the constant type family
  \begin{align*}
    P(\ctxlock{L}, i) :\defeq \left(\prd{t : \Tiny} A\substuck{t}{L}\right) + \left(\prd{t : \Tiny} B\substuck{t}{L}\right)
  \end{align*}
  The cases we have to provide are:
  \begin{alignat*}{2}
    &\mathsf{unsplit}_{A + B}(\ctxlock{L}, \lambda t. \inl(\rget{l}(t))) &&:\defeq \inl(\rget{l}) \\
    &\mathsf{unsplit}_{A + B}(\ctxlock{L}, \lambda t. \inr(\rget{r}(t))) &&:\defeq \inr(\rget{r})
  \end{alignat*}
  In one direction, the composite computes definitionally:
  \begin{align*}
    \mathsf{unsplit}_{A + B}(\mathsf{split}_{A + B}(\inl(f)))
    \defeq \mathsf{unsplit}_{A + B}(\lambda i. \inl(f(i)))
    \defeq \inl(f)
  \end{align*}
  In the other direction, we use higher-dimensional pattern matching again, this time with motive
  \begin{align*}
    P'(\ctxlock{L}, f) :\defeq \prd{t:\Tiny} \mathsf{split}_{A\substuck{t}{L} + B\substuck{t}{L}}\left(\mathsf{unsplit}_{A\substuck{t}{L} + B\substuck{t}{L}}(f)\right)(t) = f(t)
  \end{align*}
  In the branches, the left-hand side of the path computes away:
  \begin{align*}
    \mathsf{split}(\mathsf{unsplit}(\lambda t. \inl(\rget{l}(t)))) 
    \defeq \mathsf{split}(\inl(\rget{l}))
    \defeq \lambda t. \inl(\rget{l}(t))
  \end{align*}
  and similarly for $\inr$, so we may inhabit the type family by
  \begin{alignat*}{2}
    \mathsf{inv}_{A + B}(\ctxlock{L}, \lambda t. \inl(\rget{l}(t)), t) &:\defeq \refl{\inl(\rget{l}(t))} \\
    \mathsf{inv}_{A + B}(\ctxlock{L}, \lambda t. \inr(\rget{r}(t)), t) &:\defeq \refl{\inr(\rget{r}(t))}    
  \end{alignat*}
  proving (via function extensionality) that $\mathsf{unsplit}_{A+B}$ is
  indeed an inverse.
\end{proof}

Notice that we did not need the input types $A$ and $B$ to be guarded by
$\rformu$ in the above Proposition. With a little more care, we could
obtain a more refined map
\begin{align*}
  \rform{L}\left(\left((\prd{t:\Tiny} A\substuck{t}{L}) + (\prd{t:\Tiny} B\substuck{t}{L})\right) \to (\prd{t:\Tiny} A\substuck{t}{L} + B\substuck{t}{L})\right)
\end{align*}
whence the original equivalence may be extracted.

\subsection{Transpension}\label{sec:transpension}

An equivalent characterisation of a tiny object in a 1-topos, due to
Freyd~\cite[Proposition 1.2]{yetter:tiny}, is an object $\Tiny$ such
that the dependent product functor
$\Pi_\Tiny : \mathcal{E}/\Tiny \to \mathcal{E}$ has a right adjoint
$\nabla_\Tiny : \mathcal{E} \to \mathcal{E}/\Tiny$. This operation was
given the name \emph{transpension} in~\cite{dependable-atomicity,
  transpension}.

If $\Tiny$ is tiny in the ordinary sense, then in a 1-topos this
adjoint $\nabla_\Tiny$ can be defined by sending each object $B$ to
the left map in the pullback
\[
  \begin{tikzcd}
    \bullet \ar[r] \ar[d, "\nabla_\Tiny B", swap] \arrow[dr, phantom, "\lrcorner", very near start] & \rformu \left(\sm{P : \Omega} P \to B \right) \ar[d, "\rformu \proj_1"] \\
    \Tiny \ar[r, "{(\chi_{\id_\Tiny})^\vee}" swap] & \rformu \Omega
  \end{tikzcd}
\]
In the univalent setting this construction is no longer correct, but
simply replacing $\Omega$ with the universe works: for
$B : \rformu \univ$, define the type family
$\rdepform{-}{B} : \Tiny \to \univ$ by the pullback
\[
  \begin{tikzcd}
    \sm{t : \Tiny} \rdepform{t}{B} \ar[r] \ar[d] \arrow[dr, phantom, "\lrcorner", very near start] & \rformu \left(\sm{X : \univ} X \to \rget{B} \right) \ar[d, "\rformu \proj_1"] \\
    \Tiny \ar[r, "{(-=\id_\Tiny)^\vee}" swap] & \rformu \univ
  \end{tikzcd}
\]
where the map along the bottom is the transpose of
$(\lambda f. f = \id_\Tiny) : (\Tiny \to \Tiny) \to \univ$, namely
\begin{align*}
  t \mapsto \rintro{K}{\prd{s : \Tiny} t\substuck{s}{K} = s}
\end{align*}

We can simplify this description a fair amount. Because $\rformu$
commutes with $\Sigma$-types, the map on the right is equivalent to
the projection
\begin{align*}
  \proj_1 : \sm{X : \rformu \univ} \rform{K}{(\rget{X} \to \rget{B})} \to \rformu \univ
\end{align*}
The fibre of the pullback over a fixed $t : \Tiny$ is then easy to
calculate: it is equivalent to the type of the second component of the sum:
\begin{align*}
  &\rform{L}{(\rget{X} \to \rget{B})}[(\rintro{K}{\prd{s : \Tiny} t\substuck{s}{K} = s})/X] \\
  &\defeq \rform{L}{\left( (\prd{s : \Tiny} t\substuck{s}{L} = s) \to \rget{B} \right)}
\end{align*}
and so we take this as our definition.
\begin{definition}
  For $B : \rformu \univ$, define the \emph{transpension} $B^{1/-} : \Tiny \to \univ$ to be
  \begin{align*}
    \rdepform{t}{B} :\defeq \rform{L}{\left( (\prd{s : \Tiny} t\substuck{s}{L} = s) \to \rget{B} \right)}
  \end{align*}
\end{definition}

The whole point of this construction was to be a right adjoint to the dependent product with $\Tiny$, and indeed it is:
\begin{proposition}
  For $A : \Tiny \to \univ$ and $B : \rformu \univ$
  \begin{align*}
    \rform{L} \left( \left(\prd{s : \Tiny} A\substuck{s}{L}(s)\right) \to \rget{B} \right) \equiv \left(\prd{t : \Tiny} A(t) \to \rdepform{t}{B} \right)
  \end{align*}
\end{proposition}
\begin{proof}
  To apply dependent adjointness (\cref{prop:dep-adj}) we
  need to massage the interior of the $\rform{L}$ a little. Let $P$
  denote the type family
  \begin{align*}
    &P : \left(\sm{t : \Tiny} A(t)\right) \to \rformu{\univ} \\
    &P((t, a), \ctxlock{K}) :\defeq (\prd{s:\Tiny} t\substuck{s}{K} t = s) \to \rget{B}
  \end{align*}
  Then
  \begin{align*}
    &\rform{L} \left(\left(\prd{s:\Tiny} (A\substuck{s}{L})(s)\right) \to \rget{B} \right) \\
    &\equiv \rform{L} \left(\left(\prd{s:\Tiny}\sm{t : \Tiny} {a : A\substuck{s}{L}(t)} (t = s) \right)\to \rget{B} \right) && \text{(Contractible pair)} \\
    &\equiv \rform{L} \left(\prd{f:\prd{t':\Tiny} \sm{t : \Tiny} A\substuck{t'}{L}(t)} \left(\prd{s:\Tiny} \proj_1(f(i)) = s\right) \to \rget{B} \right)  && \text{(Univ.\ property of $\Sigma$)} \\
    &\defeq \rform{L} \left(\prd{f:\prd{t':\Tiny} \sm{t : \Tiny} A\substuck{t'}{L}(t)} P\substuck{i}{L}(f(i), \rbind{i})\right) && \text{(Definition of $P$)} \\
    &\equiv \prd{(t,a) : \sm{t : \Tiny} A(t)} \rform{L}{(P\substuck{i}{L})(t\substuck{i}{L},a\substuck{i}{L}, \rbind{i})} && \text{(\cref{prop:dep-adj})} \\
    &\defeq \prd{(t,a) : \sm{t : \Tiny} A(t)} \rform{L}{\left((\prd{s : \Tiny} t \substuck{s}{L} = s) \to \rget{B}\right)} && \text{(Definition of $P$)} \\
    &\defeq \prd{(t,a) : \sm{t : \Tiny} A(t)} \rdepform{t}{B} && \text{(Definition of $B^{1/t}$)} \\
    &\equiv\prd{t : \Tiny} A(t) \to \rdepform{t}{B} && \text{(Currying)}
  \end{align*}
\end{proof}

\begin{corollary}
  $\rformu \rget{B} \equiv \left(\prd{t : \Tiny} \rdepform{t}{B} \right)$
\end{corollary}
\begin{proof}
  Plug in $A(t) :\defeq 1$ in the above.
\end{proof}

A surprising fact is that the pullback of $\Tiny$ to any slice
$\mathcal{E}/X$ is itself also tiny~\cite[Theorem 1.4]{yetter:tiny},
but crucially, the right adjoint witnessing the tininess of
$\Tiny \times X$ in $\mathcal{E}/X$ is not simply the pullback of
$\rformsym$.

By a similar strategy as for transpension, we can describe the correct
right adjoint internally. Given a family $D : X \to \univ$ over a
closed type $X$, this right adjoint
$\rformslicesym{-} D : X \to \univ$ (to coin some notation) may be
calculated by
\[
  \rformslicesym{x} D :\defeq \rform{L}{\left(\prd{x' : X} \left( \prd{k : \Tiny} x' = x\substuck{k}{L}\right) \to D(x')\right)}
\]
We have not yet encountered a use for this construction.

\section{Applications}\label{sec:applications}

We end with some commentary on possible applications of our theory.

\paragraph{Synthetic Differential Geometry.} In future work with David
Jaz Myers, we intend to construct form classifiers of various kinds
internal to type theory and prove the properties asserted
axiomatically in \cite[\S 3]{david:modal-fracture} and \cite[Appendix
A]{david:microlinear}. In particular, we plan to give a synthetic
construction of the de Rham complex by constructing a long exact sequence
of form classifiers $\RR \to \Lambda^1 \to \Lambda^2 \to \cdots$.

In this setting there are many important tiny objects, for example,
not only is the infinitesimal interval $D$ tiny, but also each
``infinitesimal patch'' around the origin in $\RR^n$:
\begin{align*}
  D(n) :\defeq \{ x : \RR^n \mid \forall i j \,\, x_i x_j = 0 \}
\end{align*}

Our new judgemental rules and type former have a straightforward
extension to multiple tiny objects, with each a closed type. Each
$\ctxlock{}$ and $\rformu{}$ would be annotated with the corresponding
tiny object, and the rules for the $\rformu{}$-type would be
specialised for each tiny object.

We take this opportunity to propose some notation to be used when
there are multiple tiny objects at hand. On account of the copattern
notation, we suggest writing the amazing right adjoint as
$\Tiny \looparrowright A$ rather than ${}^\Tiny\mkern-12mu\rformu{A}$
or similar. This would associate to the right like ordinary function
types so that, for example, the unit of the adjunction would be
written
\begin{align*}
  &\eta_A : A \to \Tiny \looparrowright \Tiny \to A \\
  &\eta_A(a, \ctxlock{L}, t) :\defeq a \substuck{t}{L}
\end{align*}
with the argument list aligning nicely with the structure of the type.

\paragraph{Internal Models of Cubical Type Theory.} The present type
theory removes the necessity for axioms in the construction of a
universe of fibrations in~\cite[Theorem 5.2]{lops}. Let us briefly
repeat the construction to demonstrate where things become more
explicit.

The ambient type theory is intensional, predicative MLTT with function
extensionality and UIP: we therefore denote the ambient universe
$\Set$ rather than $\univ$ and aim to construct a new universe
$\univfib$ that classifies fibrations. There is an interval type
$\Int$ which is assumed to be tiny.

The main input data is a \emph{notion of composition structure}, which
is a function $\Compstr : (\Int \to \Set) \to \Set$ that parameterises
the notion of fibration:
\begin{align*}
  &\isFib : \prd{\Gamma : \Set}{A : \Gamma \to \Set} \Set \\
  &\isFib(\Gamma)(A) :\defeq \prd{p : \Int \to \Gamma} \Compstr(A \circ p) \\
  &\Fib : \Set \to \Set \\
  &\Fib(\Gamma) :\defeq \sm{A : \Gamma \to \Set} \isFib(\Gamma)(A)
\end{align*}


The construction proceeds the same way as in~\cite[Theorem 5.2]{lops},
with the universe classifying `crisp' fibrations constructed as the
pullback
\[
  \begin{tikzcd}
    \univfib \ar[r] \ar[d] \arrow[dr, phantom, "\lrcorner", very near start] & \rformu \sm{A : \Set} A \ar[d, "\rformu \proj_1"] \\
    \Set \ar[r, "\Compstr^\vee" swap] & \rformu \Set
  \end{tikzcd}
\]
We can simplify this description by commuting $\rformsym$ with
$\Sigma$ and substituting, a move not available when working with
the axiomatically asserted tininess. A more convenient definition of
$\univfib$ is then
\begin{align*}
  \univfib
  &:\defeq \sm{X : \Set} \rformu{\rget{A}} [\Compstr^\vee(X)/A] \\
  &\defeq \sm{X : \Set} \rform{L}{\Compstr(\lambda j. X\substuck{j}{L})}
\end{align*}



But inspecting this construction, things are much simpler if we
slightly change our notion of fibration. Surrounding $\Compstr$ with
$\rformsym$ in the body of $\isFib$, define:
\begin{align*}
  &\isNewFib : \prd{\Gamma : \Set}{A : \Gamma \to \Set} \Set \\
  &\isNewFib(\Gamma)(A) :\defeq \prd{g : \Gamma} \rform{L}{\Compstr(\lambda i. A\substuck{i}{L}(g\substuck{i}{L}))} \\
  &\NewFib : \Set \to \Set \\
  &\NewFib(\Gamma) :\defeq \sm{A : \Gamma \to \Set} \isNewFib(\Gamma)(A)
\end{align*}

Unchanged, the type $\univfib$ defined above classifies
\emph{non-crisp} fibrations of this tweaked kind, in that there is a
completely unrestricted equivalence between the type of functions
$\Gamma \to \univfib$ and $\NewFib(\Gamma)$.
We can construct an element of $\NewFib \univfib$ tautologically,
running the identity $\univfib \to \univfib$ backwards through that
equivalence. The proofs of fibrancy for each type constructor would
need to be re-done to account for the presence of $\rformsym$ in
$\isNewFib$, something we also leave for future work. There is the
intriguing prospect of these proofs amounting to a constructive
\emph{implementation} of cubical type theory internal to MLTT with a
tiny object, if enough of the non-computational axioms used
in~\cite{lops} can be avoided.

\paragraph{Higher-Order Abstract Syntax.}

\newcommand{\Viny}{\mathbb{V}}
\newcommand{\iTm}{\mathtt{Tm}}
\newcommand{\ivar}{\mathtt{var}}
\newcommand{\ilam}{\mathtt{lam}}
\newcommand{\iapp}{\mathtt{app}}
\newcommand{\isubst}{\mathsf{subst}}
\newcommand{\iconj}{\mathsf{conj}}
\newcommand{\isplit}{\mathsf{splkit}}

I am indebted to Jon Sterling for pointing out the following
application. In~\cite{fpt:abstract-syntax, ft:name-value-passing}, the
authors work in the internal language of the category of functors
$\mathsf{FinSet} \to \Set$ to give semantics for higher-order abstract
syntax. The inclusion $\Viny : \mathsf{FinSet} \to \Set$ is thought of
as an object of \emph{abstract variables}. It is tiny, because it is
represented by $1^\op \in \mathsf{FinSet}^\op$.




Suppose we have an inductive type $\iTm$ representing the abstract
syntax of the untyped $\lambda$-calculus, and want to implement
single-variable substitution
$\isubst : (\Viny \to \iTm) \to \iTm \to \iTm$ so that
$\isubst(\lambda i. v, u)$ substitutes $u$ for $i$ wherever it occurs
in $v$. We need to induct on $v$ ``under the binder'' for $i$, and the
tininess of $\Viny$ allows us to do this by similar techniques
to~\cref{sec:induction}. Applying adjointness, it will be enough to
define the conjugate operation with type
$\iTm \to \rformu{(\iTm \to \iTm)}$, which can be done by ordinary
induction on $\iTm$.

\paragraph{Implementation.} Finally, we have a prototype
implementation of a type-checker available at
\url{https://github.com/mvr/tiny}, using a novel extension of
normalisation-by-evaluation to handle the new context structure and
type former. The key insight is that the semantic environment
corresponding to a context $\Gamma, \ctxlock{L}$ can be a
\emph{meta-level function} from semantic values of type $\Tiny$ to semantic
environments for the context $\Gamma$. When looking up a variable in
the environment, the values attached to the variable as keys are used
as the arguments to these functions. Our type-checker verifies that
the two round-trips in \cref{prop:dep-adj} are indeed proven by
$\refl{}$ in both directions
(\url{https://github.com/mvr/tiny/blob/master/adjunction.tiny}). The
normalisation algorithm works in practice, but it is left for future
work to verify that it works in theory.

\printbibliography

\appendix

\section{Proofs}\label{sec:proofs}

In this appendix, we provide proofs that the admissible rules described in
\cref{sec:contexts} preserve typing.

Establishing metatheoretic properties of our type theory is somewhat
easier than modal theories like, for example, \cite{gsb:implementing}
and \cite{rfl:spec-note}, because the new rules only ever extend the
context on the right rather than performing some manipulation
potentially at any point in the middle. In particular, it is not
necessary to prove versions of `lock strengthening/weakening/exchange'
as in \cite{gsb:implementing}, or `marking/unmarking' operations as in
\cite{rfl:spec-note}.

Throughout, we continue to let $\subkeyo{\vec{t}}{L}$ stand in for a sequence such
as $\subkeyoe{t_1}{\lockn{L}_1}, \dots, \subkeyoe{t_n}{\lockn{L}_n}$.

Like substitution, our new admissible rules are definable purely on
well-scoped syntax.

\begin{definition}
  A \emph{scope} is an ordered list of variable names and lock names,
  where the variable names have no associated types. Unlike ordinary
  type theory, the relative position of variables and locks is
  important. Every ordinary context and telescope has an underlying
  scope.

  The judgement $G \yields r \rawterm$ denotes a raw term in scope
  $G$. These are defined as usual, but importantly, in a well-scoped
  term, all variable uses have the correct number of associated keys
  for their position in the scope.
\end{definition}

\begin{definition}
  The unit and counit operations
  \begin{mathpar}
    \inferrule*[left=counit,fraction={-{\,-\,}-}]
    {G, G'' \yields r \rawterm \and G, G' \yields t \rawterm
      \text{ for } \lockn{L} \in \locksin{G'} }
    {G, G', G'' \yields r \admbra{\subkeyo{\vec{t}}{L}} \rawterm} \and
    \inferrule*[left=unit,fraction={-{\,-\,}-}]
    {G, i, \ctxlock{L}, G' \yields r \rawterm}
    {G, G' \yields r\sublock{i}{L} \rawterm}
  \end{mathpar}
  are defined on raw syntax as described in \cref{sec:contexts}.
\end{definition}

These operations on raw syntax satisfy a number of associativity
properties whose mind-numbing proofs we delay until later.

\begin{proposition}
  The counit operation is admissible on context telescopes, types and terms.
  \begin{mathpar}
    \inferrule*[left=counit-tele,fraction={-{\,-\,}-}]
    {\Gamma \yields \Gamma'' \tele \and \Gamma, \Gamma' \yields t : \Tiny
      \text{ for } \lockn{L} \in \locksin{\Gamma'} }
    {\Gamma, \Gamma' \yields \Gamma'' \admbra{\subkeyo{\vec{t}}{L}} \tele} \and
    \inferrule*[left=counit,fraction={-{\,-\,}-}]
    {\Gamma, \Gamma'' \yields a : A \and \Gamma, \Gamma' \yields t : \Tiny
      \text{ for } \lockn{L} \in \locksin{\Gamma'} }
    {\Gamma, \Gamma', \Gamma'' \admbra{\subkeyo{\vec{t}}{L}} \yields a \admbra{\subkeyo{\vec{t}}{L}} : A \admbra{\subkeyo{\vec{t}}{L}}}    
  \end{mathpar}
\end{proposition}
\begin{proof}
  On telescopes:

  \begin{casesp}
  \item
    \begin{mathpar}
      \inferrule*{\Gamma \yields \Gamma'' \tele \and \Gamma, \Gamma'' \yields A \type}{\Gamma \yields \Gamma'', x : A \tele}
    \end{mathpar}
  By induction $\Gamma, \Gamma' \yields \Gamma''\admbra{\subkeyo{\vec{t}}{L}} \tele$ and $\Gamma, \Gamma', \Gamma''\admbra{\subkeyo{\vec{t}}{L}} \yields A\admbra{\subkeyo{\vec{t}}{L}} \type$, so
  $\Gamma, \Gamma' \yields \Gamma''\admbra{\subkeyo{\vec{t}}{L}}, A\admbra{\subkeyo{\vec{t}}{L}} \tele$

  \item   \begin{mathpar}
    \inferrule*{\Gamma \yields \Gamma'' \tele}{\Gamma \yields \Gamma'', \ctxlock{L} \tele}
  \end{mathpar}
  By induction $\Gamma, \Gamma' \yields \Gamma''\admbra{\subkeyo{\vec{t}}{L}} \tele$ so $\Gamma, \Gamma' \yields \Gamma''\admbra{\subkeyo{\vec{t}}{L}}, \ctxlock{L} \tele$.
  \end{casesp}

  On terms:
  \begin{casesp}
  \item
  \begin{mathpar}
    \inferrule*
    {\Gamma_1, x : A, \Gamma_2, \Gamma'' \yields s : \Tiny \text{ for } \lockn{J} \in \locksin{\Gamma_2} \\\\ 
     \Gamma_1, x : A, \Gamma_2, \Gamma'' \yields s'' : \Tiny \text{ for } \lockn{K} \in \locksin{\Gamma''} }
    {\Gamma_1, x : A, \Gamma_2, \Gamma'' \yields x \stubra{\subkeyo{\vec{s}}{J}, \subkeyo{\vec{s''}}{K}} : A\admbra{\subkeyo{\vec{s}}{J}, \subkeyo{\vec{s''}}{K}}}  
  \end{mathpar}
  Inductively, we have
  \begin{align*}
    \Gamma_1, x : A, \Gamma_2, \Gamma', \Gamma''\admbra{\subkeyo{\vec{t}}{L}} \yields \vec{s}\admbra{\subkeyo{\vec{t}}{L}} : \Tiny \\
    \Gamma_1, x : A, \Gamma_2, \Gamma', \Gamma''\admbra{\subkeyo{\vec{t}}{L}} \yields \vec{s''}\admbra{\subkeyo{\vec{t}}{L}} : \Tiny
  \end{align*}
  Also inductively, we may use the latter to form
  \begin{align*}
    \Gamma_1, x : A, \Gamma_2, \Gamma', \Gamma''\admbra{\subkeyo{\vec{t}}{L}} \yields \vec{t}\admbra{\subkeyo{\vec{s''}\admbra{\subkeyo{\vec{t}}{L}}}{K}} : \Tiny    
  \end{align*}
  These are all the keys required to form
  \begin{align*}
    \Gamma_1, x : A, \Gamma_2, \Gamma', \Gamma''\admbra{\subkeyo{\vec{t}}{L}} &\yields x \stubra{\subkeyo{\vec{s}\admbra{\subkeyo{\vec{t}}{L}}}{J}, \vec{t}\admbra{\subkeyo{\vec{s''}\admbra{\subkeyo{\vec{t}}{L}}}{K}}, \subkeyo{\vec{s''}\admbra{\subkeyo{\vec{t}}{L}}}{K}} \\ &: A\admbra{\subkeyo{\vec{s}\admbra{\subkeyo{\vec{t}}{L}}}{J}, \vec{t}\admbra{\subkeyo{\vec{s''}\admbra{\subkeyo{\vec{t}}{L}}}{K}}, \subkeyo{\vec{s''}\admbra{\subkeyo{\vec{t}}{L}}}{K}}
  \end{align*}
  And finally $A\admbra{\subkeyo{\vec{s}\admbra{\subkeyo{\vec{t}}{L}}}{J}, \subkeyo{\vec{t}\admbra{\subkeyo{\vec{s''}\admbra{\subkeyo{\vec{t}}{L}}}{K}}}{L}, \subkeyo{\vec{s''}\admbra{\subkeyo{\vec{t}}{L}}}{K}} \defeq A\admbra{\subkeyo{\vec{s}}{J}, \subkeyo{\vec{s''}}{K}}\admbra{\subkeyo{\vec{t}}{L}}$ by \cref{lem:counit-assoc}.

  \item
  \begin{mathpar}
    \inferrule*
    {\Gamma, \Gamma''_1, x : A, \Gamma''_2 \yields s'' : \Tiny \text{ for } \lockn{K} \in \locksin{\Gamma''_2}}
    {\Gamma, \Gamma''_1, x : A, \Gamma''_2 \yields x \stubra{\subkeyo{\vec{s''}}{K}} : A\admbra{\subkeyo{\vec{s''}}{K}}}  
  \end{mathpar}
  Inductively $\Gamma, \Gamma', \Gamma''_1\admbra{\subkeyo{\vec{t}}{L}}, x : A\admbra{\subkeyo{\vec{t}}{L}}, \Gamma''_2\admbra{\subkeyo{\vec{t}}{L}} \yields s''\admbra{\subkeyo{\vec{t}}{L}} : \Tiny$, and we may form $\Gamma, \Gamma', \Gamma''_1\admbra{\subkeyo{\vec{t}}{L}}, x : A\admbra{\subkeyo{\vec{t}}{L}}, \Gamma''_2\admbra{\subkeyo{\vec{t}}{L}} \yields x \stubra{\subkeyo{\vec{s''}\admbra{\subkeyo{\vec{t}}{L}}}{K}} : A\admbra{\subkeyo{\vec{t}}{L}}\admbra{\subkeyo{\vec{s''}\admbra{\subkeyo{\vec{t}}{L}}}{K}}$. Finally, $A\admbra{\subkeyo{\vec{t}}{L}}\admbra{\subkeyo{\vec{s''}\admbra{\subkeyo{\vec{t}}{L}}}{K}} \defeq A\admbra{\subkeyo{\vec{s''}}{K}}\admbra{\subkeyo{\vec{t}}{L}}$ by \cref{lem:counit-assoc-right}.

  \item
  \begin{mathpar}
    \inferrule*
    {\Gamma, \Gamma'' \yields A \type \and \Gamma, \Gamma'', x : A \yields B \type}
    {\Gamma, \Gamma'' \yields \prd{x : A} B \type}  
  \end{mathpar}
  
  Inductively $\Gamma, \Gamma', \Gamma'' \admbra{\subkeyo{\vec{t}}{L}} \yields A \admbra{\subkeyo{\vec{t}}{L}} \type$ and also, with the telescope $\Gamma''$ extended, $\Gamma, \Gamma', \Gamma'' \admbra{\subkeyo{\vec{t}}{L}}, x : A \admbra{\subkeyo{\vec{t}}{L}} \yields B\admbra{\subkeyo{\vec{t}}{L}} \type$, and we may form $\Gamma, \Gamma', \Gamma'' \admbra{\subkeyo{\vec{t}}{L}} \yields \prd{x:A \admbra{\subkeyo{\vec{t}}{L}}} B \admbra{\subkeyo{\vec{t}}{L}} \type$.

  \item
  \begin{mathpar}
    \inferrule*
    {\Gamma, \Gamma'', \ctxlock{K} \yields A \type}
    {\Gamma, \Gamma'' \yields \rform{K} A \type}  
  \end{mathpar}  
  Inductively $\Gamma, \Gamma', \Gamma'' \admbra{\subkeyo{\vec{t}}{L}}, \ctxlock{K} \yields A \admbra{\subkeyo{\vec{t}}{L}} \type$ and so we may form $\Gamma, \Gamma', \Gamma'' \admbra{\subkeyo{\vec{t}}{L}} \yields \rform{K} (A \admbra{\subkeyo{\vec{t}}{L}}) \type$.

  \item
  \begin{mathpar}
    \inferrule*
    {\Gamma, \Gamma'', \ctxlock{K} \yields a : A}
    {\Gamma, \Gamma'' \yields \rintro{K}{a} : \rform{K} A}  
  \end{mathpar}  
  Similar: Inductively $\Gamma, \Gamma', \Gamma'' \admbra{\subkeyo{\vec{t}}{L}}, \ctxlock{K} \yields a \admbra{\subkeyo{\vec{t}}{L}} : A \admbra{\subkeyo{\vec{t}}{L}}$ and so we may form $\Gamma, \Gamma', \Gamma'' \admbra{\subkeyo{\vec{t}}{L}} \yields \rintro{K}{a \admbra{\subkeyo{\vec{t}}{L}}} : \rform{K} (A \admbra{\subkeyo{\vec{t}}{L}})$.
  
  \item
  \begin{mathpar}
    \inferrule*
    {\Gamma, \Gamma'', i : \Tiny \yields r : \rform{K} A}
    {\Gamma, \Gamma'' \yields r(\rbind i) : A\sublock{i}{K} \type}  
  \end{mathpar}  
  Inductively $\Gamma, \Gamma', \Gamma'' \admbra{\subkeyo{\vec{t}}{L}}, i : \Tiny \yields r \admbra{\subkeyo{\vec{t}}{L}} : \rform{K}(A)\admbra{\subkeyo{\vec{t}}{L}}$. By definition $(\rform{K}A)\admbra{\subkeyo{\vec{t}}{L}} \defeq \rform{K}(A\admbra{\subkeyo{\vec{t}}{L}})$, and so reapplying the rule gives $\Gamma, \Gamma', \Gamma'' \admbra{\subkeyo{\vec{t}}{L}} \yields (r \admbra{\subkeyo{\vec{t}}{L}})(\rbind i) : A \admbra{\subkeyo{\vec{t}}{L}}\sublock{i}{K}$. Finally, $A \admbra{\subkeyo{\vec{t}}{L}}\sublock{i}{K} \defeq A \sublock{i}{K}\admbra{\subkeyo{\vec{t}}{L}}$ by \cref{lem:unit-counit-commute}.
  \end{casesp}
\end{proof}

\begin{proposition}
  The unit operation is admissible on context telescopes, types and terms.
  \begin{mathpar}
    \inferrule*[left=unit-tele,fraction={-{\,-\,}-}]
    {\Gamma, i : \Tiny, \ctxlock{L} \yields \Gamma' \tele}
    {\Gamma \yields \Gamma'\sublock{i}{L} \tele} \and
    \inferrule*[left=unit,fraction={-{\,-\,}-}]
    {\Gamma, i : \Tiny, \ctxlock{L}, \Gamma' \yields a : A}
    {\Gamma, \Gamma'\sublock{i}{L} \yields a\sublock{i}{L} : A\sublock{i}{L}}
  \end{mathpar}
\end{proposition}
\begin{proof}
  On telescopes:
  \begin{casesp}
  \item
    \begin{mathpar}
      \inferrule*
      {\Gamma, i : \Tiny, \ctxlock{L} \yields \Gamma' \tele \and \Gamma, i : \Tiny, \ctxlock{L}, \Gamma' \yields A \type}
      {\Gamma, i : \Tiny, \ctxlock{L} \yields \Gamma', x : A \tele}
    \end{mathpar}
    Inductively $\Gamma \yields \Gamma'\sublock{i}{L} \tele$ and $\Gamma, \Gamma'\sublock{i}{L} \yields A\sublock{i}{L} \type$, and we may form $\Gamma \yields \Gamma'\sublock{i}{L}, x : A\sublock{i}{L} \tele$.

  \item   \begin{mathpar}
      \inferrule*
      {\Gamma, i : \Tiny, \ctxlock{L} \yields \Gamma' \tele}
      {\Gamma, i : \Tiny, \ctxlock{L} \yields \Gamma', \ctxlock{K} \tele}
    \end{mathpar}
    Inductively $\Gamma \yields \Gamma'\sublock{i}{L} \tele$ so also $\Gamma \yields \Gamma'\sublock{i}{L}, \ctxlock{K} \tele$.
  \end{casesp}

  On terms:
  \begin{casesp}
  \item When $x \ndefeq i$ and $x$ comes before $i$ in the context:
    \begin{mathpar}
    \inferrule*
    {\Gamma_1, x : A, \Gamma_2, i : \Tiny, \ctxlock{L}, \Gamma' \yields s : \Tiny \text{ for } \lockn{J} \in \locksin{\Gamma_2} \\\\
     \Gamma_1, x : A, \Gamma_2, i : \Tiny, \ctxlock{L}, \Gamma' \yields l : \Tiny \\\\      
     \Gamma_1, x : A, \Gamma_2, i : \Tiny, \ctxlock{L}, \Gamma' \yields s' : \Tiny \text{ for } \lockn{K} \in \locksin{\Gamma'}}
    {\Gamma_1, x : A, \Gamma_2, i : \Tiny, \ctxlock{L}, \Gamma' \yields x \stubra{\subkeyo{\vec{s}}{J}, \subkeyo{l}{L}, \subkeyo{\vec{s'}}{K}} : A\admbra{\subkeyo{\vec{s}}{J}, \subkeyo{l}{L}, \subkeyo{\vec{s'}}{K}}}  
  \end{mathpar}

  Inductively, $\Gamma_1, x : A, \Gamma_2, \Gamma'\sublock{i}{L} \yields s\sublock{i}{L} : \Tiny$ and $\Gamma_1, x : A, \Gamma_2, \Gamma'\sublock{i}{L} \yields s'\sublock{i}{L} : \Tiny$, so
  \begin{align*}
    \Gamma_1, x : A, \Gamma_2, \Gamma'\sublock{i}{L} &\yields x \stubra{\subkeyo{\vec{s}\sublock{i}{L}}{J}, \subkeyo{\vec{s'}\sublock{i}{L}}{K}} \\
                                                     &: A\admbra{\subkeyo{\vec{s}\sublock{i}{L}}{J}, \subkeyo{\vec{s'}\sublock{i}{L}}{K}} 
  \end{align*}
  Finally, $A\admbra{\subkeyo{\vec{s}\sublock{i}{L}}{J}, \subkeyo{\vec{s'}\sublock{i}{L}}{K}} \defeq A\admbra{\subkeyo{\vec{s}}{J}, \subkeyo{l}{L}, \subkeyo{\vec{s'}}{K}}\sublock{i}{L}$, by \cref{lem:unit-noccur}.  

    \item When $x \ndefeq i$ and $x$ comes after $i$ in the context:
    \begin{mathpar}
    \inferrule*
    {\Gamma, i : \Tiny, \ctxlock{L}, \Gamma'_1, x : A, \Gamma'_2 \yields s' : \Tiny \text{ for } \lockn{K} \in \locksin{\Gamma'_2}}
    {\Gamma, i : \Tiny, \ctxlock{L}, \Gamma'_1, x : A, \Gamma'_2 \yields x \stubra{\subkeyo{\vec{s'}}{K}} : A\admbra{\subkeyo{\vec{s'}}{K}}}  
  \end{mathpar}
  Inductively, $\Gamma, \Gamma'_1\sublock{i}{L}, x : A\sublock{i}{L}, \Gamma'_2\sublock{i}{L} \yields \vec{s'}\sublock{i}{L} : \Tiny$ and so $\Gamma, \Gamma'_1\sublock{i}{L}, x : A\sublock{i}{L}, \Gamma'_2\sublock{i}{L} \yields x\stubra{\subkeyo{\vec{s'}\sublock{i}{L}}{K}} : A\admbra{\subkeyo{\vec{s'}\sublock{i}{L}}{K}}$. Finally, by Lemma~\ref{lem:unit-counit-commute}, $A\sublock{i}{L}\admbra{\subkeyo{\vec{s'}\sublock{i}{L}}{K}} \defeq A\admbra{\subkeyo{\vec{s'}}{K}}\sublock{i}{L}$.

\item
    \begin{mathpar}
    \inferrule*
    {\Gamma, i : \Tiny, \ctxlock{L}, \Gamma' \yields l : \Tiny \\\\
      \Gamma, i : \Tiny, \ctxlock{L}, \Gamma' \yields s' : \Tiny \text{ for } \lockn{K} \in \locksin{\Gamma'}}
    {\Gamma, i : \Tiny, \ctxlock{L}, \Gamma' \yields i \stubra{\subkeyo{l}{L}, \subkeyo{\vec{s'}}{K}} : \Tiny}
  \end{mathpar}
  Inductively, $\Gamma, \Gamma'\sublock{i}{L} \yields l\sublock{i}{L} : \Tiny$ and we are done.

  \item
  \begin{mathpar}
    \inferrule*
    {\Gamma, i : \Tiny, \ctxlock{L}, \Gamma' \yields A \type \and \Gamma, i : \Tiny, \ctxlock{L}, \Gamma', x : A \yields B \type}
    {\Gamma, i : \Tiny, \ctxlock{L}, \Gamma' \yields \prd{x : A} B \type}  
  \end{mathpar}
  
  Inductively $\Gamma, \Gamma'\sublock{i}{L} \yields A \sublock{i}{L} \type$ and $\Gamma, \Gamma'\sublock{i}{L}, x : A\sublock{i}{L} \yields B\sublock{i}{L} \type$, so we may form $\Gamma, \Gamma'\sublock{i}{L} \yields \prd{x:A \sublock{i}{L}} B\sublock{i}{L} \type$.

  \item
  \begin{mathpar}
    \inferrule*
    {\Gamma, i : \Tiny, \ctxlock{L}, \Gamma', \ctxlock{K} \yields A \type}
    {\Gamma, i : \Tiny, \ctxlock{L}, \Gamma' \yields \rform{K} A \type}  
  \end{mathpar}  
  Inductively $\Gamma, \Gamma'\sublock{i}{L}, \ctxlock{K} \yields A\sublock{i}{L} \type$ and so we may form $\Gamma, \Gamma'\sublock{i}{L} \yields \rform{K} (A \sublock{i}{L}) \type$.

  \item
  \begin{mathpar}
    \inferrule*
    {\Gamma, i : \Tiny, \ctxlock{L}, \Gamma', \ctxlock{K} \yields a : A}
    {\Gamma, i : \Tiny, \ctxlock{L}, \Gamma' \yields \rintro{K}{a} : \rform{K} A}  
  \end{mathpar}  
  Similar: Inductively $\Gamma, \Gamma'\sublock{i}{L}, \ctxlock{K} \yields a\sublock{i}{L} : A\sublock{i}{L}$ and so we may form $\Gamma, \Gamma'\sublock{i}{L} \yields \rintro{K}{a\sublock{i}{L}} : \rform{K} (A \sublock{i}{L}) \type$.
  
  \item
  \begin{mathpar}
    \inferrule*
    {\Gamma, i : \Tiny, \ctxlock{L}, \Gamma', k : \Tiny \yields r : \rform{K} A}
    {\Gamma, i : \Tiny, \ctxlock{L}, \Gamma' \yields r(\rbind k) : A\sublock{k}{K} \type}  
  \end{mathpar}  
  Inductively $\Gamma, \Gamma'\sublock{i}{L}, k : \Tiny \yields r \sublock{i}{L} : \rform{K}(A)\sublock{i}{L}$. By definition $\rform{K}(A)\sublock{i}{L} \defeq \rform{K}(A\sublock{i}{L})$, and so reapplying the rule gives $\Gamma, \Gamma'\sublock{i}{L} \yields (r \sublock{i}{L})(\rbind k) : A \sublock{i}{L}\sublock{k}{K}$. Finally, $A \sublock{i}{L}\sublock{k}{K} \defeq A \sublock{k}{K}\sublock{i}{L}$ by \cref{lem:unit-unit-commute}.  
\end{casesp}
\end{proof}

Now, the proofs of the equations used in the previous two
propositions, all of which hold on the level of well-scoped raw terms.
For each equation, the variable rule is the only case where something
interesting happens.

\begin{lemma}\label{lem:counit-assoc}
  \begin{mathpar} 
    \inferrule*[fraction={-{\,-\,}-}]
    {G_1 \yields \judge \rawterm \\\\
     G_1, G_2, G'' \yields s \rawterm \text{ for } \lockn{J} \in \locksin{G_2} \\\\ 
      G_1, G_2, G' \yields t \rawterm \text{ for } \lockn{L} \in \locksin{G'} \\\\ 
     G_1, G_2, G'' \yields s'' \rawterm \text{ for } \lockn{K} \in \locksin{G''}       
    }
    {G_1, G_2, G', G'' \yields \judge\admbra{\subkeyo{\vec{s}}{J}, \subkeyo{\vec{s''}}{K}}\admbra{\subkeyo{\vec{t}}{L}} \defeq
\judge\admbra{\subkeyo{\vec{s}\admbra{\subkeyo{\vec{t}}{L}}}{J}, \subkeyo{\vec{t}\admbra{\subkeyo{\vec{s''}\admbra{\subkeyo{\vec{t}}{L}}}{K}}}{L}, \subkeyo{\vec{s''}\admbra{\subkeyo{\vec{t}}{L}}}{K}}}
  \end{mathpar}
\end{lemma}
\begin{proof}
  Suppose $G_1 \defeq G_{11}, x, G_{12}$ and the variable usage is
  \begin{mathpar}
    \inferrule*
    {G_{11}, x, G_{12} \yields m \rawterm \text{ for } \lockn{M} \in \locksin{G_{12}}}
    {G_{11}, x, G_{12} \yields x\substuck{\vec{m}}{M} \rawterm}
  \end{mathpar}
  Then
  \[
  \begin{aligned}
    &x\stubra{\subkeyo{\vec{m}}{M}}\admbra{\subkeyo{\vec{s}}{J}, \subkeyo{\vec{s''}}{K}}\admbra{\subkeyo{\vec{t}}{L}} \\
    &\defeq x\stubra{\subkeyo{\vec{m}\admbra{\subkeyo{\vec{s}}{J}, \subkeyo{\vec{s''}}{K}}}{M}, \subkeyo{\vec{s}}{J}, \subkeyo{\vec{s''}}{K}}\admbra{\subkeyo{\vec{t}}{L}} \\
    &\defeq x\stubra{\subkeyo{\vec{m}\admbra{\subkeyo{\vec{s}}{J}, \subkeyo{\vec{s''}}{K}}\admbra{\subkeyo{\vec{t}}{L}}}{M}, \subkeyo{\vec{s}\admbra{\subkeyo{\vec{t}}{L}}}{J}, \subkeyo{\vec{t}\admbra{\subkeyo{\vec{s''}\admbra{\subkeyo{\vec{t}}{L}}}{K}}}{L}, \subkeyo{\vec{s''}\admbra{\subkeyo{\vec{t}}{L}}}{K}} \\
    &\defeq x\stubra{
      \begin{aligned}[t]
        &\subkeyo{\vec{m}\admbra{\subkeyo{\vec{s}\admbra{\subkeyo{\vec{t}}{L}}}{J}, \subkeyo{\vec{t}\admbra{\subkeyo{\vec{s''}\admbra{\subkeyo{\vec{t}}{L}}}{K}}}{L}, \subkeyo{\vec{s''}\admbra{\subkeyo{\vec{t}}{L}}}{K}}}{M}, \\
        &\subkeyo{\vec{s}\admbra{\subkeyo{\vec{t}}{L}}}{J}, \\
        &\subkeyo{\vec{t}\admbra{\subkeyo{\vec{s''}\admbra{\subkeyo{\vec{t}}{L}}}{K}}}{L}, \\
        &\subkeyo{\vec{s''}\admbra{\subkeyo{\vec{t}}{L}}}{K}}
      \end{aligned} \\
    &\defeq x\stubra{\subkeyo{\vec{m}}{M}}\admbra{\subkeyo{\vec{s}\admbra{\subkeyo{\vec{t}}{L}}}{J}, \subkeyo{\vec{t}\admbra{\subkeyo{\vec{s''}\admbra{\subkeyo{\vec{t}}{L}}}{K}}}{L}, \subkeyo{\vec{s''}\admbra{\subkeyo{\vec{t}}{L}}}{K}}
  \end{aligned}
  \]
  by inductive hypothesis on $m$.
\end{proof}

\begin{lemma}\label{lem:counit-assoc-right}
  \begin{mathpar}
    \inferrule*[fraction={-{\,-\,}-}]{
      G, G''_1 \yields \judge \rawterm \\\\
      G, G' \yields t \rawterm \text{ for } \lockn{L} \in \locksin{G'} \\\\
      G, G''_1, G''_2 \yields s'' \rawterm \text{ for } \lockn{K} \in \locksin{G''_2}
    }{G, G', G''_1, G''_2 \yields \judge\admbra{\subkeyo{\vec{s''}}{K}}\admbra{\subkeyo{\vec{t}}{L}} \defeq \judge \admbra{\subkeyo{\vec{t}}{L}}\admbra{\subkeyo{\vec{s''}\admbra{\subkeyo{\vec{t}}{L}}}{K}}}
  \end{mathpar}
\end{lemma}
\begin{proof}
  There are two places where the variable may lie.
  \begin{casesp}
  \item If $x$ lies in $G$, so $G \defeq G_1, x, G_2$:
    \begin{mathpar}
      \inferrule*
      {G_1, x, G_2, G''_1 \yields m \rawterm \text{ for } \lockn{M} \in \locksin{G_2} \\\\
       G_1, x, G_2, G''_1 \yields n \rawterm \text{ for } \lockn{N} \in \locksin{G''_1}}
      {G_1, x, G_2, G''_1 \yields x\stubra{\subkeyo{\vec{m}}{M}, \subkeyo{\vec{n}}{N}} \rawterm}
    \end{mathpar}
    And then
    \[
      \begin{aligned}
      &x\stubra{\subkeyo{\vec{m}}{M}, \subkeyo{\vec{n}}{N}}\admbra{\subkeyo{\vec{s''}}{K}}\admbra{\subkeyo{\vec{t}}{L}} \\
      &\defeq x\stubra{\subkeyo{\vec{m}\admbra{\subkeyo{\vec{s''}}{K}}}{M}, \subkeyo{\vec{n}\admbra{\subkeyo{\vec{s''}}{K}}}{N}, \subkeyo{\vec{s''}}{K}}\admbra{\subkeyo{\vec{t}}{L}} \\
      &\defeq x\stubra{
        \begin{aligned}[t]
          &\subkeyo{\vec{m}\admbra{\subkeyo{\vec{s''}}{K}}\admbra{\subkeyo{\vec{t}}{L}}}{M}, \\
          &\subkeyo{\vec{t}\admbra{\subkeyo{\vec{n}\admbra{\subkeyo{\vec{s''}}{K}}\admbra{\subkeyo{\vec{t}}{L}}}{N}, \subkeyo{\vec{s''}\admbra{\subkeyo{\vec{t}}{L}}}{K}}}{L}, \\
          & \subkeyo{\vec{n}\admbra{\subkeyo{\vec{s''}}{K}}\admbra{\subkeyo{\vec{t}}{L}}}{N}, \\
          & \subkeyo{\vec{s''}\admbra{\subkeyo{\vec{t}}{L}}}{K}} 
        \end{aligned}
      \end{aligned}
    \]
    Each of these keys attached to $x$ can be rewritten by induction:
    \begin{align*}
      \vec{m}\admbra{\subkeyo{\vec{s''}}{K}}\admbra{\subkeyo{\vec{t}}{L}} &\defeq \vec{m}\admbra{\subkeyo{\vec{t}}{L}}\admbra{\subkeyo{\vec{s''}\admbra{\subkeyo{\vec{t}}{L}}}{K}} \\
      \vec{t}\admbra{\subkeyo{\vec{n}\admbra{\subkeyo{\vec{s''}}{K}}\admbra{\subkeyo{\vec{t}}{L}}}{N}, \subkeyo{\vec{s''}\admbra{\subkeyo{\vec{t}}{L}}}{K}}
&\defeq\vec{t}\admbra{\subkeyo{\vec{n}\admbra{\subkeyo{\vec{t}}{L}}\admbra{\subkeyo{\vec{s''}\admbra{\subkeyo{\vec{t}}{L}}}{K}}}{N}, \subkeyo{\vec{s''}\admbra{\subkeyo{\vec{t}}{L}}}{K}} \\
      \vec{n}\admbra{\subkeyo{\vec{s''}}{K}}\admbra{\subkeyo{\vec{t}}{L}} &\defeq \vec{n}\admbra{\subkeyo{\vec{t}}{L}}\admbra{\subkeyo{\vec{s''}\admbra{\subkeyo{\vec{t}}{L}}}{K}}
    \end{align*}
    And so, continuing the chain of equations:
    \[
      \begin{aligned}
      &x\stubra{\subkeyo{\vec{m}}{M}, \subkeyo{\vec{n}}{N}}\admbra{\subkeyo{\vec{s''}}{K}}\admbra{\subkeyo{\vec{t}}{L}} \\
      &\defeq \dots \\
      &\defeq x\stubra{
        \begin{aligned}[t]
          &\subkeyo{\vec{m}\admbra{\subkeyo{\vec{t}}{L}}\admbra{\subkeyo{\vec{s''}\admbra{\subkeyo{\vec{t}}{L}}}{K}}}{M}, \\
          &\subkeyo{\vec{t}\admbra{\subkeyo{\vec{n}\admbra{\subkeyo{\vec{t}}{L}}\admbra{\subkeyo{\vec{s''}\admbra{\subkeyo{\vec{t}}{L}}}{K}}}{N}, \subkeyo{\vec{s''}\admbra{\subkeyo{\vec{t}}{L}}}{K}}}{L}, \\
          &\subkeyo{\vec{n}\admbra{\subkeyo{\vec{t}}{L}}\admbra{\subkeyo{\vec{s''}\admbra{\subkeyo{\vec{t}}{L}}}{K}}}{N}, \\
          &\subkeyo{\vec{s''}\admbra{\subkeyo{\vec{t}}{L}}}{K}} 
        \end{aligned} \\
      &\defeq x\stubra{\subkeyo{\vec{m}\admbra{\subkeyo{\vec{t}}{L}}}{M}, \subkeyo{\vec{t}\admbra{\subkeyo{\vec{n}\admbra{\subkeyo{\vec{t}}{L}}}{N}}}{L}, \subkeyo{\vec{n}\admbra{\subkeyo{\vec{t}}{L}}}{N}}\admbra{\subkeyo{\vec{s''}\admbra{\subkeyo{\vec{t}}{L}}}{K}} \\
      &\defeq x\stubra{\subkeyo{\vec{m}}{M}, \subkeyo{\vec{n}}{N}}\admbra{\subkeyo{\vec{t}}{L}}\admbra{\subkeyo{\vec{s''}\admbra{\subkeyo{\vec{t}}{L}}}{K}}
      \end{aligned}
    \]
  \item If $x$ lies in $G_1''$, so $G_1'' \defeq G_{11}'', x, G_{12}''$:
    \begin{mathpar}
      \inferrule*
      {G, G''_{11}, x, G''_{12} \yields n \rawterm \text{ for } \lockn{N} \in \locksin{G''_{12}}}
      {G, G''_{11}, x, G''_{12} \yields x\stubra{\subkeyo{\vec{n}}{N}} \rawterm}
    \end{mathpar}
    The reasoning is essentially the same:
    \[
      \begin{aligned}
      &x\stubra{\subkeyo{\vec{n}}{N}}\admbra{\subkeyo{\vec{s''}}{K}}\admbra{\subkeyo{\vec{t}}{L}} \\
      &\defeq x\stubra{\subkeyo{\vec{n}\admbra{\subkeyo{\vec{s''}}{K}}}{N}, \subkeyo{\vec{s''}}{K}} \admbra{\subkeyo{\vec{t}}{L}} \\
      &\defeq x\stubra{\subkeyo{\vec{t}\admbra{\subkeyo{\vec{n}\admbra{\subkeyo{\vec{s''}}{K}}\admbra{\subkeyo{\vec{t}}{L}}}{N}, \subkeyo{\vec{s''}\admbra{\subkeyo{\vec{t}}{L}}}{K}}}{L}, \subkeyo{\vec{n}\admbra{\subkeyo{\vec{s''}}{K}}\admbra{\subkeyo{\vec{t}}{L}}}{N}, \subkeyo{\vec{s''}\admbra{\subkeyo{\vec{t}}{L}}}{K}} \\
      &\defeq x\stubra{
        \begin{aligned}[t]
          &\subkeyo{\vec{t}\admbra{\subkeyo{\vec{n}\admbra{\subkeyo{\vec{s''}\admbra{\subkeyo{\vec{t}}{L}}}{K}}}{N}, \subkeyo{\vec{s''}\admbra{\subkeyo{\vec{t}}{L}}}{K}}}{L}, \\
          &\subkeyo{\vec{n}\admbra{\subkeyo{\vec{s''}\admbra{\subkeyo{\vec{t}}{L}}}{K}}}{N}, \\
          &\subkeyo{\vec{s''}\admbra{\subkeyo{\vec{t}}{L}}}{K}}          
        \end{aligned} \\
      &\defeq x\stubra{\subkeyo{\vec{t}\admbra{\subkeyo{\vec{n}\admbra{\subkeyo{\vec{t}}{L}}}{N}}}{L}, \subkeyo{\vec{n}\admbra{\subkeyo{\vec{t}}{L}}}{N}}\admbra{\subkeyo{\vec{s''}\admbra{\subkeyo{\vec{t}}{L}}}{K}} \\
      &\defeq x\stubra{\subkeyo{\vec{n}}{N}}\admbra{\subkeyo{\vec{t}}{L}}\admbra{\subkeyo{\vec{s''}\admbra{\subkeyo{\vec{t}}{L}}}{K}}
      \end{aligned}
    \]
    again by induction in the middle step.
  \end{casesp}
\end{proof}

\begin{lemma}\label{lem:unit-counit-commute}
  The unit and counit commute, depending on the relative position in
  the context the two operations are being applied to.
  \begin{mathpar}
    \inferrule*[fraction={-{\,-\,}-}]
    {G, G''_1, i, \ctxlock{K}, G''_2 \yields \judge \rawterm \\\\
     G, G' \yields t \rawterm \text{ for } \lockn{L} \in \locksin{G'}}
    {G, G', G''_1, G''_2 \yields \judge \admbra{\subkeyo{\vec{t}}{L}}\sublock{i}{K} \defeq \judge\sublock{i}{K}\admbra{\subkeyo{\vec{t}}{L}}}  
  \end{mathpar}

  \begin{mathpar}
    \inferrule*[fraction={-{\,-\,}-}]
    {G_1, i, \ctxlock{L}, G_2, G'' \yields \judge \rawterm \\\\
     G_1, i, \ctxlock{L}, G_2, G' \yields s' \rawterm \text{ for } \lockn{K} \in \locksin{G'}}
   {G_1, i, \ctxlock{L}, G_2, G', G'' \yields \judge\admbra{\subkeyo{\vec{s'}}{K}}\sublock{i}{L} \defeq \judge\sublock{i}{L}\admbra{\subkeyo{\vec{s'}\sublock{i}{L}}{K}} }
 \end{mathpar}  

\end{lemma}
\begin{proof}
  For the first, there are several places a variable could be, but the case of
  interest is the variable rule for $i$:
  \begin{mathpar}
    \inferrule*
    {G, G''_1, i, \ctxlock{K}, G''_2 \yields k \rawterm \text{ for } \lockn{K} \\\\
      G, G''_1, i, \ctxlock{K}, G''_2 \yields o \rawterm \text{ for } \lockn{O} \in \locksin{G''_2} 
    }
    {G, G''_1, i, \ctxlock{K}, G''_2 \yields i\stubra{\subkeyo{k}{K}, \subkeyo{\vec{o}}{O}} \rawterm}
  \end{mathpar}
  in which case
  \begin{align*}
    &i\stubra{\subkeyo{k}{K}, \subkeyo{\vec{o}}{O}}\admbra{\subkeyo{\vec{t}}{L}}\sublock{i}{K} \\
    &\defeq i\stubra{\subkeyo{\vec{t}\admbra{\subkeyo{k\admbra{\subkeyo{\vec{t}}{L}}}{K}, \subkeyo{\vec{o}\admbra{\subkeyo{\vec{t}}{L}}}{O}}}{L}, \subkeyo{k\admbra{\subkeyo{\vec{t}}{L}}}{K}, \subkeyo{\vec{o}\admbra{\subkeyo{\vec{t}}{L}}}{O}}\sublock{i}{K} \\
    &\defeq k\admbra{\subkeyo{\vec{t}}{L}}\sublock{i}{K} \\
    &\defeq k\sublock{i}{K}\admbra{\subkeyo{\vec{t}}{L}} \\
    &\defeq i\stubra{\subkeyo{k}{K}, \subkeyo{\vec{o}}{O}}\sublock{i}{K}\admbra{\subkeyo{\vec{t}}{L}} 
  \end{align*}

  For the second, again the variable rule for $i$ is the interesting case:
  \begin{mathpar}
    \inferrule*[fraction={-{\,-\,}-}]
    {G_1, i, \ctxlock{L}, G_2, G'' \yields l \rawterm \text{ for } \lockn{L} \\\\
      G_1, i, \ctxlock{L}, G_2, G'' \yields m \rawterm \text{ for } \lockn{M} \in \locksin{G_2} \\\\
      G_1, i, \ctxlock{L}, G_2, G'' \yields n \rawterm \text{ for } \lockn{N} \in \locksin{G''}       
    }
    {G_1, i, \ctxlock{L}, G_2, G'' \yields i\stubra{\subkeyo{l}{L}, \subkeyo{\vec{m}}{M}, \subkeyo{\vec{n}}{N}} \rawterm}
  \end{mathpar}
  in which case
  \begin{align*}
    &i\stubra{\subkeyo{l}{L}, \subkeyo{\vec{m}}{M}, \subkeyo{\vec{n}}{N}}\admbra{\subkeyo{\vec{s'}}{K}}\sublock{i}{L} \\
    &\defeq i\stubra{\subkeyo{l\admbra{\subkeyo{\vec{s'}}{K}}}{L}, \subkeyo{\vec{m}\admbra{\subkeyo{\vec{s'}}{K}}}{M}, \subkeyo{\vec{s'}\admbra{\subkeyo{\vec{n}\admbra{\subkeyo{\vec{s'}}{K}}}{N}}}{K}, \subkeyo{\vec{n}\admbra{\subkeyo{\vec{s'}}{K}}}{N}}\sublock{i}{L} \\
    &\defeq l\admbra{\subkeyo{\vec{s'}}{K}}\sublock{i}{L} \\
    &\defeq l\sublock{i}{L}\admbra{\subkeyo{\vec{s'}\sublock{i}{L}}{K}} \\
    &\defeq i\stubra{\subkeyo{l}{L}, \subkeyo{\vec{m}}{M}, \subkeyo{\vec{n}}{N}}\sublock{i}{L}\admbra{\subkeyo{\vec{s'}\sublock{i}{L}}{K}}
  \end{align*}
  
\end{proof}

\begin{lemma}\label{lem:unit-noccur}
  \begin{mathpar}
    \inferrule*[fraction={-{\,-\,}-}]
    {G_1 \yields \judge \rawterm \\\\
      G_1, G_2, i, \ctxlock{L}, G' \yields s \rawterm \text{ for } \lockn{M} \in \locksin{G_2} \\\\
      G_1, G_2, i, \ctxlock{L}, G' \yields l \rawterm \text{ for } \lockn{L} \\\\
      G_1, G_2, i, \ctxlock{L}, G' \yields s' \rawterm \text{ for } \lockn{N} \in \locksin{G'}
 }{G_1, G_2, G' \yields \judge\admbra{\subkeyo{\vec{s}}{J}, \subkeyo{l}{L}, \subkeyo{\vec{s'}}{K}}\sublock{i}{L} \defeq \judge\admbra{\subkeyo{\vec{s}\sublock{i}{L}}{J}, \subkeyo{\vec{s'}\sublock{i}{L}}{K}}}
  \end{mathpar}  
\end{lemma}
\begin{proof}
  Straightforward induction, key is that $i$ does not occur in $\judge$.
\end{proof}

\begin{lemma}\label{lem:unit-unit-commute}
  \begin{mathpar}
    \inferrule*[fraction={-{\,-\,}-}]
    {G, i, \ctxlock{L}, G', k, \ctxlock{K}, G'' \yields \judge \rawterm}
    {G, G', G'' \yields \judge \sublock{i}{L}\sublock{k}{K} \defeq \judge \sublock{k}{K}\sublock{i}{L}}
  \end{mathpar}
\end{lemma}
\begin{proof}
  The interesting cases are, of course, the variable rules for $i$ and $k$:
  \begin{casesp}
  \item
    \begin{mathpar}
    \inferrule*
    {G, i, \ctxlock{L}, G', k, \ctxlock{K}, G'' \yields t \rawterm \text{ for } \lockn{L} \\\\
     G, i, \ctxlock{L}, G', k, \ctxlock{K}, G'' \yields m \rawterm \text{ for } \lockn{M} \in \locksin{G'} \\\\
     G, i, \ctxlock{L}, G', k, \ctxlock{K}, G'' \yields s \rawterm \text{ for } \lockn{K} \\\\
     G, i, \ctxlock{L}, G', k, \ctxlock{K}, G'' \yields n \rawterm \text{ for } \lockn{N} \in \locksin{G''}
    }
    {G, i, \ctxlock{L}, G', k, \ctxlock{K}, G'' \yields i\stubra{\subkeyo{t}{L}, \subkeyo{\vec{m}}{M}, \subkeyo{s}{K}, \subkeyo{\vec{n}}{N}} \rawterm}
  \end{mathpar}
  Then:
  \begin{align*}
    &i\stubra{\subkeyo{t}{L}, \subkeyo{\vec{m}}{M}, \subkeyo{s}{K}, \subkeyo{\vec{n}}{N}}\sublock{i}{L}\sublock{k}{K} \\
    &\defeq t\sublock{i}{L}\sublock{k}{K} \\
    &\defeq t\sublock{k}{K}\sublock{i}{L} \\
    &\defeq i\stubra{\subkeyo{t\sublock{k}{K}}{L}, \subkeyo{\vec{m}\sublock{k}{K}}{M}, \subkeyo{\vec{n}\sublock{k}{K}}{N}}\sublock{i}{L} \\
    &\defeq i\stubra{\subkeyo{t}{L}, \subkeyo{\vec{m}}{M}, \subkeyo{s}{K}, \subkeyo{\vec{n}}{N}}\sublock{k}{K}\sublock{i}{L}
  \end{align*}
  \item
    \begin{mathpar}
    \inferrule*
    {G, i, \ctxlock{L}, G', k, \ctxlock{K}, G'' \yields s \rawterm \text{ for } \lockn{K} \\\\
     G, i, \ctxlock{L}, G', k, \ctxlock{K}, G'' \yields n \rawterm \text{ for } \lockn{N} \in \locksin{G''}
    }
    {G, i, \ctxlock{L}, G', k, \ctxlock{K}, G'' \yields k\stubra{\subkeyo{s}{K}, \subkeyo{\vec{n}}{N}} \rawterm}
  \end{mathpar}
  Then:
  \begin{align*}
    &k\stubra{\subkeyo{s}{K}, \subkeyo{\vec{n}}{N}}\sublock{i}{L}\sublock{k}{K} \\
    &\defeq k\stubra{\subkeyo{s\sublock{i}{L}}{K}, \subkeyo{\vec{n}\sublock{i}{L}}{N}}\sublock{k}{K} \\
    &\defeq s\sublock{i}{L}\sublock{k}{K} \\
    &\defeq s\sublock{k}{K}\sublock{i}{L} \\
    &\defeq k\stubra{\subkeyo{s}{K}, \subkeyo{\vec{n}}{N}}\sublock{k}{K}\sublock{i}{L}
  \end{align*}  
  \end{casesp}
\end{proof}

\end{document}